\documentclass[12pt,reqno]{amsart}

\usepackage[utf8]{inputenc}
\usepackage{amssymb}
\usepackage{amsmath}
\usepackage{amsthm}
\usepackage{amsfonts}
\usepackage{bbm}
\usepackage{bookmark}
\usepackage{hyperref}
\hypersetup{hidelinks, pdfstartview={FitH}}
\usepackage{color}
\usepackage{xcolor}
\usepackage{mathrsfs} 
\usepackage[normalem]{ulem}
\usepackage{caption}
\usepackage{graphicx}
\usepackage{tikz}

\theoremstyle{plain}
\newtheorem{theorem}{Theorem}[section]
\newtheorem{lemma}[theorem]{Lemma}

\newtheorem{corollary}[theorem]{Corollary}
\newtheorem{proposition}[theorem]{Proposition}

\theoremstyle{definition}

\theoremstyle{remark}

\numberwithin{equation}{section}
\newtheorem*{theorem*}{Theorem} 

\newcommand{\Z}{{\mathbb Z}}
\newcommand{\R}{{\mathbb R}}

\newcommand{\N}{{\mathbb N}}
\newcommand{\C}{{\mathbb C}}

\DeclareMathOperator{\dist}{dist}

\DeclareMathOperator{\size}{size}

\DeclareFontFamily{U}{mathx}{\hyphenchar\font45}
\DeclareFontShape{U}{mathx}{m}{n}{
<5> <6> <7> <8> <9> <10>
<10.95> <12> <14.4> <17.28> <20.74> <24.88>
mathx10
}{}
\DeclareSymbolFont{mathx}{U}{mathx}{m}{n}
\DeclareFontSubstitution{U}{mathx}{m}{n}
\DeclareMathAccent{\widecheck}{0}{mathx}{"71}

\author[Lars Becker]{Lars Becker}
\address{Mathematics Department, Princeton University, Princeton, NJ 08544, USA}
\email{lbecker@math.princeton.edu}

\author{Polona Durcik}
\address{Schmid College of Science and Technology, Chapman University, 
One University Drive, Orange, CA 92866, USA}
\email{durcik@chapman.edu}

\begin{document}

\title{The shifted bilinear Hilbert transform}

\begin{abstract}
    We prove $L^p$  estimates for the shifted bilinear Hilbert transform, with a polylogarithmic bound in the size of the shift. As applications, we obtain $r$-variation estimates for bilinear ergodic averages in the sharp range $r > 2$, a sharp bilinear Hörmander multiplier theorem, and a $\log$-Dini theorem for bilinear singular integrals.
\end{abstract}

\maketitle

\section{Introduction}

This paper is about improved quantitative estimates related to the bilinear Hilbert transform  
\begin{equation}\label{e:BHT}
    \int_\R f_1(x-y) f_2(x+y) \frac{1}{y} \, dy\,.
\end{equation}
Our main result is a family of shifted bilinear Hilbert transform estimates. 
As our main application we obtain sharp pointwise variation estimates for double recurrence averages in ergodic theory. 
In addition, we prove new optimal low-regularity kernel and multiplier theorems  for the bilinear Hilbert transform.

\subsection{The shifted bilinear Hilbert transform}
For a tempered distribution $\varphi$, we define the bilinear average of Schwartz functions $f_1,f_2$ by
\begin{equation}\label{e:bil_avg}
    B_t(\varphi,f_1,f_2)(x) = \int_{\R}f_1(x-y)f_2(x+y)t^{-1}\varphi(t^{-1}y)dy\,.
\end{equation}
For a function $\psi$ of mean  zero we consider operators of the form 
\begin{equation}\label{e:operator}
    \operatorname{BHT}_\psi(f_1, f_2) =
    \sum_{s \in \mathbb{Z}} B_{2^s}(\psi, f_1, f_2)\,.
\end{equation}
They are natural variants of the bilinear Hilbert transform \eqref{e:BHT}. In fact, with a suitable choice of $\psi$ the operator $\operatorname{BHT}_\psi$ coincides with \eqref{e:BHT}. The proof of boundedness of the bilinear Hilbert transform by Lacey and Thiele  \cite{LaceyThieleL2, LaceyThiele} extends to $\operatorname{BHT}_\psi$, showing that it defines a bounded operator 
\[
    L^{p_1}(\R) \times L^{p_2}(\R) \to L^{p}(\R)
\]
whenever $\psi$ is a mean zero Schwartz function and the exponents satisfy 
\begin{equation}\label{e:p_aspt}
    \frac{1}{p} = \frac{1}{p_1} + \frac{1}{p_2}, \qquad\qquad 1< p_1, p_2 \le \infty, \qquad\qquad  \frac{2}{3} < p<\infty\,.
\end{equation}
The norm estimate for $\operatorname{BHT}_\psi$ implicit in 
\cite{LaceyThieleL2, LaceyThiele} 
depends on high-order Schwartz seminorms of $\psi$ and is far from optimal for many choices of $\psi$. 
Our main result is an improved estimate when $\psi$ is a shift of a localized bump function. Inspired by the folklore name for analogous estimates for linear singular integrals and maximal functions, we refer to this as the shifted bilinear Hilbert transform estimate.

Let $\mathcal{S}_0$ denote the set of all mean-zero functions $\psi$ satisfying
\[
     \sup_x |x|^m |\psi^{(n)}(x)| \le 1\,, \qquad\qquad 0 \le n, m \le 100\,.
\]
For $\tau>0$, define the set of  $\tau$-shifted bump functions 
\[
    \mathcal{S}_0^{\tau} = \{T_{\tau} \psi \ : \ \psi \in \mathcal{S}_0\}, \qquad\qquad T_\tau \psi(x) = \psi(x - \tau)\,.
\]
Our main result is the following estimate for a variant of the operator \eqref{e:operator} where $\psi$ is replaced by a sequence of shifted bump functions $\psi_s$.
The shift is allowed to vary mildly with the scale $s$, but is still restricted to an interval $[2^m, 2^{m+1}]$.
\begin{theorem}
\label{thm:main2}
    Let $p, p_1, p_2$ be exponents satisfying \eqref{e:p_aspt}. There exists $C > 0$ such that the following holds.
    Let $m \ge 0$ and let $(\tau_s)_{s \in \mathbb{Z}}$  be a sequence of integers with  
    \[
        2^{m-1} \le |\tau_s| \le 2^{m}\,.
    \]
    Let $(\psi_s)_{s \in \mathbb{Z}}$  be a sequence   with $\psi_s \in \mathcal{S}_0^{\tau_s}$. Then for all Schwartz functions $f_1,f_2:\R\to \C$, 
    \[
        \Big\|\sum_{s \in \mathbb{Z}} B_{2^s}(\psi_s, f_1, f_2)\Big\|_{L^{p}(\R)} \le C m^{4} \|f_1\|_{L^{p_1}(\R)} \|f_2\|_{L^{p_2}(\R)}\,.
    \] 
\end{theorem}

If $\psi \in \mathcal{S}_0^\tau$, the Schwartz semi-norms of $\psi$  typically grow polynomially in  $\tau$,  and hence exponentially in $m$. Consequently, the operator norm bound for the operators $\textup{BHT}_\psi$ implicit in \cite{LaceyThieleL2, LaceyThiele} is of the order $\tau^{O(1)} = 2^{O(m)}$. Theorem \ref{thm:main2} improves this to $m^{4}$.

This improved constant is somewhat analogous to classical results for linear singular integrals, often called the shifted maximal function and the shifted square function, see the textbooks \cite[Chapter 2.5.10]{Steinbook}, \cite[Theorems 4.5 and 4.6]{MS2013}, or Nagel and Stein \cite{NagelStein84} where this seems to have been originally observed. There, the operator norm bound can also be reduced from exponential to linear in $m$. While elementary, these improvements have proven to be extremely useful technical tools. To give just three examples, they enter into Sj\"ogren's proof of pointwise convergence of Poisson integrals at the boundary in symmetric spaces  \cite{Sjogren86}, rough kernel singular integral estimates by Dosidis, Park, and Slav{\'i}kov{\'a}, \cite{dosidis2026} and the recent  breakthrough works by Krause, Mirek, and Tao \cite{KMT} and by Kosz, Mirek, Peluse, Wan, and Wright \cite{KMPWW} in pointwise ergodic theory.

The proof of Theorem \ref{thm:main2} will be given in Sections \ref{s:overview} and \ref{s:BHT}. It is a modification of the standard line of argument in time-frequency analysis from \cite{LaceyThieleL2, LaceyThiele}; we loosely follow the presentation in \cite{T2006}. First, the shifted bilinear Hilbert transform is expressed as a superposition of certain combinatorially defined dyadic model forms. The combinatorics in our proof are slightly unconventional, taking into account the shift. The model form is then decomposed efficiently into so-called trees. A single tree is easily estimated, and by orthogonality arguments it is proved that these estimates can be summed. These orthogonality arguments are where 
most of our new work happens. 
The shift 
leads to a new notion of shifted trees, which enjoy less orthogonality than the classical trees in \cite{LaceyThieleL2, LaceyThiele}. To still prove estimates we split at several places into small scales, where classical arguments apply, and $O(m)$ many large scales where we use ad-hoc arguments. This results in several $O(m)$ losses.  

We now  present three consequences of Theorem \ref{thm:main2}.

\subsection{Sharp variation estimates for  ergodic averages} 
\label{ss:ergodic}
Let  $(X, \mathcal{B}, \mu)$ be a probability space and let $T:X \to X$ be  an invertible, bi-measurable, measure-preserving transformation. We consider the double recurrence averages 
\begin{equation} \label{e:dbl_avg}
    M_n(f_1,f_2)(x) = \frac{1}{n}\sum_{i=0}^{n-1}f_1(T^ix)f_2(T^{-i}x)\,.
\end{equation}
Bourgain \cite{Bourgain1990} extended Birkhoff's pointwise ergodic theorem to these averages: The limit of $M_n(f_1, f_2)$ as $n \to \infty$ exists almost surely, for all $f_1, f_2 \in L^\infty(X, \mu)$. Bourgain's work was motivated by Furstenberg's study of multiple recurrence \cite{Furstenberg81}. The pointwise convergence of triple recurrence averages and more general multiple recurrence averages remain difficult open problems.

We are  interested in a quantitative strengthening of Bourgain's result. For $1\le r<\infty$, the $r$-variation of a function $a:\Omega\to \C$ with $\Omega\subset \R$  is defined by 
\[\|a\|_{V^r(\Omega)} = \|a(t)\|_{V_t^r(\Omega)} = \sup_{\substack{J\in \N\\ t_0,t_1,\ldots, t_J\in \Omega\\t_0<t_1<\ldots <t_J}}\Big( \sum_{j=1}^J|a(t_j)-a(t_{j-1})|^r\Big)^{1/r}\,.\]
If $a: \mathbb{N} \to \mathbb{C}$ is a sequence with $\|a\|_{V^r(\mathbb{N})} < \infty$, then $a$ converges. Thus, an estimate for the $r$-variation norm provides a quantitative strengthening of the qualitative assertion of pointwise convergence. Such variation estimates are known for many classical convergence problems in harmonic analysis, we refer to Jones, Seeger and Wright \cite{JSW08} and the references therein. 

Our first application of Theorem \ref{thm:main2} are new $r$-variation estimates for double recurrence averages \eqref{e:dbl_avg}. For the so-called long variation, along a lacunary sequence of $n$, we obtain the following estimate with optimal range of exponents $r$.

\begin{theorem}\label{t:ergodic_long}
Let $p,p_1,p_2,$ be exponents satisfying \eqref{e:p_aspt}
and let $r > 2$.
There exists a constant $C>0$ such that for any $\sigma$-finite measure space $(X,\mathcal{B}, \mu)$,  any invertible bi-measurable measure-preserving transformation $T:X\to X$, and  any $f_1\in L^{p_1}(X)$,   $f_2\in L^{p_2}(X)$,  
\[\|M_{2^n}(f_1,f_2)(x)\|_{L_x^p(V^r_n(\N))}\le C \|f_1\|_{L^{p_1}(X)}\|f_2\|_{L^{p_2}(X)}.\]
\end{theorem}

For the so-called short variation, in between powers of $2$, our methods yield bounds in the optimal range $r > 1$ when $p \ge 2$, and a smaller range depending on $p$ when $p < 2$.

\begin{theorem}
\label{t:ergodic_short}
    Let $p,p_1,p_2,$ be exponents satisfying \eqref{e:p_aspt}
    and suppose that
    \begin{equation}\label{e:r_aspt_short}
        \frac{1}{r} < \min\Big \{\frac{3}{2} - \frac{1}{p}, 1 \Big \}\,.
    \end{equation}
    There exists a constant $C>0$ such that for any $\sigma$-finite measure space $(X,\mathcal{B}, \mu)$,  any invertible bi-measurable measure-preserving transformation $T:X\to X$, and  any $f_1\in L^{p_1}(X)$,   $f_2\in L^{p_2}(X)$,  
    \[
        \Big\|\Big(\sum_{s \in \mathbb{Z}} \|M_n(f_1,f_2)(x)\|_{V^r_n([2^s, 2^{s+1}] \cap \mathbb{N})}^2\Big)^{1/2}\Big\|_{L_x^p}\le C \|f_1\|_{L^{p_1}(X)}\|f_2\|_{L^{p_2}(X)}.
    \]
\end{theorem}

Note that the condition \eqref{e:r_aspt_short} translates for $p = 1$ still to $r > 2$, and it gives a nonempty range of $r$ for any $p > 2/3$. Combining Theorem \ref{t:ergodic_long} and Theorem \ref{t:ergodic_short} one obtains the following estimate for the full variation.

\begin{corollary}\label{c:ergodic}
    Let $p,p_1,p_2,$ be exponents satisfying \eqref{e:p_aspt} 
    and let 
    \begin{equation}\label{e:r_aspt}
        \frac{1}{r} < \min \Big \{\frac{3}{2} - \frac{1}{p}, \frac{1}{2} \Big \}\,.
    \end{equation}
    There exists a constant $C>0$ such that for any $\sigma$-finite measure space $(X,\mathcal{B}, \mu)$,  any invertible bi-measurable measure-preserving transformation $T:X\to X$, and  any $f_1\in L^{p_1}(X)$, $f_2\in L^{p_2}(X)$,  
    \[
        \|M_{n}(f_1,f_2)(x)\|_{L_x^p(V^r_n(\N))}\le C \|f_1\|_{L^{p_1}(X)}\|f_2\|_{L^{p_2}(X)}.
    \]
\end{corollary}

We note that the condition \eqref{e:r_aspt} translates to $r > 2$ whenever $p \ge 1$.

Such variational estimates for the double recurrence averages \eqref{e:dbl_avg} were previously obtained by Do, Oberlin, and Palsson \cite{DOP15}. See also earlier work of Lacey \cite{Lacey_maximal}, Demeter \cite{demeter2007}, and Demeter, Tao, and Thiele \cite{DTT2008}.  However, the results in \cite{DOP15} are not effective, they only hold for sufficiently large $r$. 
By monotonicity of $\ell^r$-sums, the $r$-variation norm of a sequence is decreasing in $r$, so   smaller values of $r$ correspond to stronger estimates.
The restriction $r > 2$ as in Theorem \ref{t:ergodic_long} and Corollary \ref{c:ergodic} is necessary even for linear averages. This can be shown  
by comparison to Brownian motion, which almost surely has infinite $2$-variation \cite{FrizVictoir}. The comparison argument is detailed by Bourgain in  \cite[Lemma 3.11]{Bourgain_arith}. 

We do not know whether the condition \eqref{e:r_aspt_short} is necessary for the short variation estimate to hold. It comes from the lack of convexity of an $\ell^{p/2}$ sum arising from the $L^p \ell^2$ norm in the short variation. This issue does not arise in \cite{DOP15}, because they only work in a higher regularity setting where this sum only has boundedly many terms. See Lemma \ref{l:square2}, the setting of \cite{DOP15} corresponds to the case $j = 0$. 

Corollary \ref{c:ergodic} can be equivalently reformulated as an estimate for the number of large jumps in the sequence $M_n(f_1, f_2)$. The $\lambda$-jump counting  function $N_\lambda(a)$ of a sequence $a: \Omega \to \mathbb{C}$ is the supremum of all integers $N$ for which there exists an increasing sequence $s_1 < t_1 \le s_2 < t_2 \le \dotsb < t_N$, all $s_i, t_i \in \Omega$, satisfying 
\[
    |a_{t_i} - a_{s_i}| > \lambda, \qquad\qquad 1\le i \le N\,.
\]
\begin{corollary}\label{c:jumps}
    In the setting of Corollary \ref{c:ergodic},   
    \[
        \| \sup_{\lambda} \lambda (N_\lambda( (M_n(f_1, f_2))_{n=1}^\infty ))^{1/r} \|_{L^p(X)} \le C \|f_1\|_{L^{p_1}(X)}\|f_2\|_{L^{p_2}(X)}\,.
    \]
\end{corollary}

Indeed, by choosing the sequence $s_1, t_1, \dotsc, t_N$, possibly dropping repeated points, in the definition of the $r$-variation one finds that
\[
    \lambda^r N_\lambda( M_n(f_1, f_2)(x) ) \le \|M_n(f_1, f_2)(x)\|_{V^r_n(\mathbb{N})}\,. 
\]
This implication can be reversed: Corollary \ref{c:jumps} with exponents $r, p, p_1, p_2$ implies Corollary \ref{c:ergodic} for the same $p, p_1, p_2$ and all $r' > r$. 
 
A further related consequence of Corollary \ref{c:ergodic} is control of the $\lambda$-entropy of the set $\{M_n(f_1, f_2)\}_{n \in \mathbb{N}}$. The $\lambda$-entropy number $E_\lambda(A)$ of a set $A \subset \R$ is the smallest number of intervals of length $2\lambda$ needed to cover $A$.  

\begin{corollary}\label{c:entropy}
    In the setting of Corollary \ref{c:ergodic},  
    \[
        \| \sup_{\lambda} \lambda (E_\lambda( \{M_n(f_1, f_2)\}_{n \in \mathbb{N}}
        ))^{1/r} \|_{L^p(X)} \le C \|f_1\|_{L^{p_1}(X)}\|f_2\|_{L^{p_2}(X)}\,.
    \]
\end{corollary}
Indeed, let $n_0 = 0$ and   
\[
    n_{j+1} = \min\{n > n_{j} \ : \ |M_n(f_1, f_2)(x) - M_{n_j}(f_1, f_2)(x)| \ge \lambda\}\,.
\]
Let $J$ be the largest index so that
$n_{J} < \infty$.
Then $E_\lambda( \{M_n(f_1, f_2)(x)\}_{n \in \mathbb{N}} ) \le J + 1$ and
\begin{align*}
    \lambda^r J &\le \sum_{j =1}^J |M_{n_j}(f_1, f_2)(x) - M_{n_{j-1}}(f_1, f_2)(x)|^r\le \|M_n(f_1, f_2)(x)\|_{V^r_n(\mathbb{N})}^r\,. 
\end{align*}
The analogue of Corollary \ref{c:entropy} for linear averages appears in \cite[Lemma 3.30]{Bourgain_arith}. 

Theorem \ref{t:ergodic_long} has similar consequences with a larger range of $r$ when one restricts attention to averages along a lacunary sequence, we do not state them here.  

We now outline how Theorems \ref{t:ergodic_long} and \ref{t:ergodic_short} follow from Theorem \ref{thm:main2}, the details will be given in Section \ref{s:ergodic}. Using standard transference arguments, see \cite[Section 2]{DOP15}, they 
follow from corresponding variation estimates for bilinear averages on $\R$.

\begin{theorem}\label{t:long}
    Let $r>2$ and $p,p_1,p_2$ be exponents satisfying \eqref{e:p_aspt}. 
There exists a constant $C>0$ such that  for any $f_1\in L^{p_1}(\R)$ and $f_2\in L^{p_2}(\R)$, 
\[\|B_{2^s}(\mathbf{1}_{[0,1]},f_1,f_2)(x)\|_{L_x^p(V^r_s(\Z))}\le C \|f_1\|_{p_1}\|f_2\|_{p_2}\,.\]
\end{theorem}

\begin{theorem}\label{t:short}
    Let $p,p_1,p_2$ be exponents satisfying \eqref{e:p_aspt} and let $r$ be such that \eqref{e:r_aspt_short} holds. 
    There exists a constant $C>0$ such that  for any $f_1\in L^{p_1}(\R)$ and $f_2\in L^{p_2}(\R)$, 
    \[
        \Big\| \Big(\sum_{s \in \mathbb{Z}} \|B_t(\mathbf{1}_{[0,1]},f_1,f_2)(x)\|_{V^r_t([2^s, 2^{s+1}])}^2\Big)^{1/2}\Big\|_{p}\le C  \|f_1\|_{p_1}\|f_2\|_{p_2}\,.
    \]
\end{theorem}

The paper \cite{DOP15} proves versions of Theorem \ref{t:long} and  Theorem \ref{t:short} where 
$\mathbf{1}_{[0,1]}$ is replaced by a smooth approximation $K$, see \cite[Theorem 1.3]{DOP15}. They then pass from $K$ back to $\mathbf{1}_{[0,1]}$ at the cost of increasing $r$. The main new difficulty here is thus the low regularity of the weight $\mathbf{1}_{[0,1]}$. 

To prove Theorem \ref{t:long}, we use \cite[Theorem 1.3]{DOP15} as a black box, which allows us to replace $\mathbf{1}_{[0,1]}$ by a mean-zero function. The mean zero property of the weight then permits us to bound the long $r$-variation by an $\ell^2$-sum, which can be treated as an $\ell^2$-valued bilinear Hilbert transform with kernel of very low regularity.  This low-regularity kernel is then  decomposed as an infinite sum of shifted kernels as in Theorem \ref{thm:main2}. At this point, the triangle inequality completes the proof. 

For Theorem \ref{t:short}, we directly use the fundamental theorem of calculus to reduce to estimating certain square functions, which are again decomposed into shifted bilinear Hilbert transforms. Here we incur a loss from the low regularity when $p < 2$, related to the failure of convexity of $\ell^{p/2}$ in that range.

It is worth noting that if the transformations $T$ and $T^{-1}$ are replaced by two general commuting measure-preserving transformations $S$ and $T$, pointwise a.e. convergence of the corresponding double averages
\begin{equation}
    \label{e:doubleavg}
    \frac{1}{n}\sum_{i=0}^{n-1}f_1(S^ix) f_2(T^ix)
\end{equation}
remains a major open problem. This is in contrast with the recent breakthroughs on pointwise convergence for polynomial averages \cite{KMT,KMPWW}. 
Convergence of the averages  \eqref{e:doubleavg} in the $L^2$ norm is known by the classical work of Conze and Lesigne \cite{CL84}. 
Sharp quantitative $L^2$  convergence was established in \cite{DKST16}. This paper  also proves sharp $L^2$ bounds for the pointwise short variation for \eqref{e:doubleavg}, and the methods therein can be used to handle the long variation for the mean-zero part in the argument outlined above, however also only in $L^2$.
Thus, combining the results of \cite{DOP15} and \cite{DKST16} one can prove at least some special cases of Theorem \ref{t:ergodic_long}, Theorem \ref{t:ergodic_short} and Corollary \ref{c:ergodic} with $p = 2$. Our methods go beyond that and establish optimal $r$-variation estimates also when $p \ne 2$, and for more general $p_1, p_2$ than what follows from \cite{DKST16}.
The approach in \cite{DKST16} relies on  the same decomposition of the characteristic function as in our proof, however, there the special form of the $L^2$ norm is exploited to obtain operators of a different nature than the bilinear Hilbert transform.

The superposition arguments we use are not specific to the characteristic function $\mathbf{1}_{[0,1]}$. 
We now present two further applications of Theorem \ref{thm:main2} where similar decompositions yield  effective bounds.

\subsection{Dini kernel regularity for the bilinear Hilbert transform}\label{ss:dini}

The natural sharp differential kernel condition necessary for Calderón-Zygmund theory is the Dini condition \cite[Chapter 1.6.5]{Steinbook}, which we recall next. 

The modulus of continuity of a uniformly continuous function $f$ is the function $\eta: [0, \infty) \to [0, \infty)$ defined by
\[
    \eta(\varepsilon) = \sup_{|x-y| \le \varepsilon} |f(x) - f(y)|\,.
\]
Every modulus of continuity satisfies 
\begin{equation}\label{e:eta_aspt}
    \eta(0) = 0, \qquad\qquad \eta(x+y) \le \eta(x) + \eta(y), \qquad\qquad \eta \ \text{is non decreasing}.
\end{equation}
An odd function $K$ is called a singular integral kernel with modulus of continuity $\eta$, or short $\eta$-kernel, if for all $x, x'$ with $2|x-x'| < |x|$,
\begin{equation}
    \label{e:K_eta}
    |\widehat K| \le \eta(1), \qquad |K(x)| \le |x|^{-1}\eta(1) , \qquad |K(x) - K(x')|\le |x|^{-1}\eta\Big(\frac{|x-x'|}{|x|}\Big).
\end{equation}
The classical   theorem of Calderón and Zygmund then states  that an $L^p$ bounded singular integral operator with kernel $K$ is of weak type $(1,1)$ provided 
\begin{equation}\label{e:classicalDini}
    \int_0^1 \eta(t) \frac{1}{t} \, dt < \infty\,. 
\end{equation}
Condition \eqref{e:classicalDini} is referred to as  the Dini condition.

Using Theorem \ref{thm:main2}, we can show that a slightly stronger condition is sufficient for boundedness of the bilinear Hilbert transform. 
We require finiteness of
\begin{equation}\label{e:dini_aspt}
    \|\eta\|_{\mathrm{Dini}} = \int_0^1 \eta(t) \frac{\lvert\log t\rvert^{4}}{t} \, dt,
\end{equation}
or, when $p < 1$, of the larger quasi-norm 
\begin{equation}\label{e:pdini_aspt}
    \|\eta\|_{\mathrm{p-Dini}} = \Big(\int_0^1 \eta(t)^p \frac{\lvert \log t\rvert^{4p}}{t^{2 - p}} \, dt\Big)^{1/p}.
\end{equation}
Below, we write $B(K,f_1,f_2) = B_1(K,f_1,f_2)$ with $B_1$ as in \eqref{e:bil_avg}. 

\begin{theorem}\label{t:Dini}
    For all   $p, p_1, p_2$ satisfying \eqref{e:p_aspt}, there exists a constant $C > 0$ such that the following holds.
    Let $\eta$ be a modulus of continuity satisfying \eqref{e:eta_aspt},  and let $K$ be an $\eta$-kernel. 
    Then for  all Schwartz function $f_1,f_2:\R\to \C$,  if $p \ge 1$ it holds that
    \[
        \|B(K,f_1,f_2)\|_{L^p(\R)} \le C \|\eta\|_{\mathrm{Dini}} \|f_1\|_{L^{p_1}(\R)} \|f_2\|_{L^{p_2}(\R)},
    \]
    and if $p < 1$, then the same estimate holds with $\|\eta\|_{\mathrm{Dini}}$ replaced by $\|\eta\|_{\mathrm{p-Dini}}$.
\end{theorem}

Note that $\|x^\alpha\|_{\mathrm{Dini}} < \infty$ if and only if $\alpha > 0$, while $\|x^\alpha\|_{\mathrm{p-Dini}} < \infty$ if and only if $\alpha > 1/p-1$. 

Thus, Theorem \ref{t:Dini} implies that any Hölder continuity of the kernel suffices for boundedness of the bilinear Hilbert transform when $p \ge 1$. The only previous work we are aware of that treats operators including the bilinear Hilbert transform with merely Hölder-continuous kernels   is \cite{Benyi+09}. Theorem \ref{t:Dini} shows that the condition in \cite{Benyi+09}  is not sharp in our setting, though  their the setup  is much more general.

When $p < 1$ we also improve upon the corresponding result of \cite{Benyi+09}. However, in this range  it is unclear whether the condition $\alpha > 1/p-1$ is necessary. This threshold matches the regularity required for boundedness of singular integrals on Hardy spaces $H^p$, see \cite[Chapter 3.3.2]{Steinbook}. However, this is likely just a relict of the method of proof using superposition arguments and shifted estimates. Indeed, the same argument applies to singular integrals in Hardy spaces, so it cannot possibly improve upon the Hardy space result. 

As in Theorem \ref{thm:main2}, we have not attempted to optimize the exponent $4$ in \eqref{e:dini_aspt} and \eqref{e:pdini_aspt}. It seems unlikely that our methods are able to give a sharp result in that direction. A perhaps simpler but still interesting question is whether there exists  a counterexample to boundedness in the $p \ge 1$ range when assuming only the classical Dini condition \eqref{e:classicalDini}, or in the $p < 1$ range with any Dini or Hölder regularity. 

\subsection{The Hörmander multiplier theorem for the bilinear Hilbert transform}\label{ss:hor}

Hörmander's multiplier theorem \cite{Hormander} provides a boundedness criterion for linear singular integral operators in terms of the regularity of the multiplier.   
Let $\varphi \in C_0^\infty(\R)$ be supported in $1/4 < |\xi| < 4$ and positive on $1/2 < |\xi| <2$.
Given a Banach space $Y$ of locally integrable functions on $\R$, we define the associated localization space $\mathscr{V}(Y)$ with norm 
\[
    \|m\|_{\mathscr{V}(Y)} = \sup_{s\in \Z} \|m(\xi) \varphi(2^{-s} \xi)\|_{Y}\,,
\]
the notation is borrowed from \cite{ConnettSchwartz}.
If $C_0^\infty(\R)$ functions define bounded Fourier multipliers on $Y$, which will be the case for all spaces considered here, then this space is independent of the choice of $\varphi$. Hörmander's multiplier theorem states that for $1 < p < \infty$ and $\sigma > 1/2$, 
\[
    \Big\| \int_\R \widehat{f}(\xi) m(\xi) e^{-2\pi ix \xi} \, d \xi \Big\|_{L^{p}(\R)} \le C \|m\|_{\mathscr{V}(H^\sigma)} \|f\|_{L^{p}(\R)}\,,
\]
where $H^\sigma$ denotes the $L^2$ based Sobolev space with smoothness $\sigma$.
In the bilinear setting, there is the Fourier inversion identity
\[
    \int_\R f_1(x - y)f_2(x + y)K(y) \, dy =  \iint_{\R^2} \widehat{f_1}(\xi) \widehat{f_2}(\eta) \widehat{K}(\xi - \eta) e^{2\pi ix(\xi + \eta)} \, d\xi \, d\eta\,,
\]
so our methods apply to bilinear multiplier operators of the form 
\[
    B(\widecheck{m}, f_1,f_2)(x) = \iint_{\R^2} \widehat{f_1}(\xi) \widehat{f_2}(\eta) m(\xi - \eta) e^{2\pi ix(\xi + \eta)} \, d\xi \, d\eta\,.
\]
Using Theorem \ref{thm:main2} we obtain the following bilinear Hörmander multiplier theorem.
\begin{theorem}\label{t:Hörmander}
    Let  $p, p_1, p_2$ satisfy  \eqref{e:p_aspt}, and let  
    \[
        \sigma > \max \Big \{\frac{1}{2}, \frac{1}{p} - \frac{1}{2} \Big \}\,. 
    \] 
    Then there  exists a constant $C> 0$ such that for  all Schwartz function $f_1,f_2:\R\to \C$, 
    \[
        \|B(\widecheck m, f_1,f_2)\|_{L^p(\R)} \le C\|m\|_{\mathscr{V}(H^\sigma)}\|f_1\|_{L^{p_1}(\R)} \|f_2\|_{L^{p_2}(\R)}\,.
    \]
\end{theorem}

In fact, we also prove a stronger theorem, with a condition lying between Hörmander's multiplier condition and his weaker integral condition, which also appears in \cite{Hormander}. Various similar conditions have appeared in the literature for linear Fourier multipliers, see the more general work of Seeger \cite{Seeger88, Seeger90} and of Carbery \cite{Carbery86}. More recently, Dosidis, Park and Slav{\'i}kov{\'a} \cite{dosidis2026} also proved boundedness of bilinear Fourier multipliers of Coifman-Meyer type under related conditions. Note however that in contrast to these references, we make no claim about sharpness of our conditions beyond the scale of Sobolev spaces $H^\sigma$. They are natural spaces one obtains from superposition arguments and Theorem \ref{thm:main2}, and we have no reason to believe that either the method or Theorem \ref{thm:main2} are sharp.

Denote for $a \ge 1$ by $A(a)$ the annulus $\{2^{a-1} \le |x| \le 2^a\}$, and denote $A(0) = (-1,1)$. For a positive sequence $(w_a)_{a\in \N}$ 
define 
\[
    \|m\|_{Y_w} = \sup_{a \ge 0}  \frac{1 +a^5}{w_a} \int_{A(a)} |\widehat m(x)|  \, dx\,,
\]
and for $p < 1$ define
\[
    \|m\|_{Y^p_w} = \sup_{a \ge 0}   \Big( \frac{1 + a^{4p+1}}{w_a^p}\int_{A(a)} |\widehat m(x)|^p  \, dx\Big)^{1/p}.
\]
\begin{theorem}\label{t:Hörmander2}
    For all $p,p_1,p_2$ as in  Theorem \ref{t:Hörmander}, there exists a constant $C > 0$ such that the following holds. For all   $w\colon\N\to [0,\infty)$ with $\|w\|_{\ell^1} \le 1$ and all Schwartz functions $f_1,f_2:\R\to \C$, if $p \ge 1$ it holds that
    \[
        \|B(\widecheck m, f_1,f_2)\|_{L^p(\R)} \le C\|m\|_{\mathscr{V}(Y_w)}\|f_1\|_{L^{p_1}(\R)} \|f_2\|_{L^{p_2}(\R)}
    \]
    and if  $p < 1$, then the same estimate holds with  $\|m\|_{\mathscr{V}(Y_w)}$ replaced by $\|m\|_{\mathscr{V}(Y_w^p)}$.
\end{theorem}

Choosing for example $w_a = a^{-2}$ and using the Cauchy-Schwarz inequality,  $H^\sigma$ embeds into $Y_w$ for all $\sigma > 1/2$ and into $Y^p_w$ for $\sigma > 1/p - 1/2$. Hence Theorem \ref{t:Hörmander} follows from Theorem \ref{t:Hörmander2}. The exponent $5$ here is one more than the exponent in Theorem \ref{thm:main2} because 
a dyadic pigeonholing step in the reduction 
introduces an additional $\log$-factor. We also note that this step is not necessary if one assumes a discrete scaling symmetry $m(\xi) = m(2\xi)$, as is done in \cite{dosidis2026}. In that case one can replace the exponent $5$ by $4$.

The condition on $\sigma$ in Theorem \ref{t:Hörmander} is sharp when $p \ge 1$, as it is already sharp for linear Fourier multipliers. For  $p < 1$,  we do not know whether the threshold is sharp, though the same remarks regarding Hardy spaces as in the previous section apply. 

Previously, Chen, Hsu, and Lin   \cite{ChenHsuLin} proved a   Hörmander multiplier theorem for more general multipliers $m(\xi, \eta)$ satisfying a suitable two-variable Hörmander condition. For multipliers of the form $m(\xi - \eta)$, their two-variable condition follows from ours, but this only recovers Theorem \ref{t:Hörmander} when $\sigma > 1$ and in a smaller range of exponents $p_1, p_2, p$. Their theorem is sharp for the more general class of multipliers they consider, which shows that for multipliers of the form $m(\xi - \eta)$, less regularity suffices  than for general multipliers $m(\xi, \eta)$.

\subsection{The shifted Carleson theorem}

We close this introduction with the remark that the method of this paper, combining shifted estimates and superposition arguments, applies also to maximal modulation operators. The following shifted version of Lie's polynomial Carleson theorem \cite{Lie} holds.

\begin{theorem}\label{t:polCarleson}
    Let $1 < p < \infty$ and $d \ge 1$. There exists $C > 0$ such that the following holds.
    Let $m \ge 0$ and let $(\tau_s)_{s \in \mathbb{Z}}$  be a sequence of integers with  
    \[
        2^{m-1} \le |\tau_s| \le 2^{m}\,.
    \]
    Let $(\psi_s)_{s \in \mathbb{Z}}$  be a sequence with $\psi_s \in \mathcal{S}_0^{\tau_s}$ and set 
    \[
        K(x) = \sum_{s \in \mathbb{Z}} 2^{-s} \psi_s(2^{-s}x)\,.
    \]
    Then for all Schwartz functions $f :\R\to \C$, 
    \[
        \Big\|\sup_{\mathrm{deg} \, P \le d} \Big| \int_{\R} f(x - y) e^{iP(y)} K(y) \, dy\Big|\Big\|_{L^{p}(\R)} \le C m^{4} \|f\|_{L^{p}(\R)}\,,
    \] 
    where the supremum is taken over all polynomials $P$ of degree at most $d$.
\end{theorem}

A slightly weaker estimate than this $L^p(\R)$ boundedness is proved in \cite{Becker24}.
Using the more careful construction of exceptional sets done in \cite{becker25} or \cite{becker+2025} it can be upgraded to an $L^p(\R)$ estimate which implies Theorem \ref{t:polCarleson}.

Repeating the superposition arguments from the proofs of Theorems \ref{t:Dini} and \ref{t:Hörmander} then yields the following low regularity versions of the polynomial Carleson theorem.

\begin{theorem}
    Let $1 < p < \infty$ and $d \ge 1$. There exists $C > 0$ such that the following holds.
    Let $\eta$ be a modulus of continuity satisfying \eqref{e:eta_aspt},  and let $K$ be an $\eta$-kernel. Then for all Schwartz functions $f$
    \[
        \Big\|\sup_{\mathrm{deg} \, P \le d} \Big| \int_{\R} f(x - y) e^{iP(y)} K(y) \,dy\Big|\Big\|_{L^{p}(\R)} \le C \|\eta\|_{\mathrm{Dini}}\|f\|_{L^{p}(\R)}\,.
    \] 
\end{theorem}

\begin{theorem}
    Let $1 < p < \infty$, $d \ge 1$ and $\sigma > 1/2$. There exists $C > 0$ such that for all Schwartz functions $f$ and all bounded functions $m: \R \to \mathbb{C}$
    \[
        \Big\|\sup_{\mathrm{deg} \, P \le d} \Big| \int_{\R} f(x - y) e^{iP(y)} \widecheck{m}(y) \, dy\Big|\Big\|_{L^{p}(\R)} \le C \|m\|_{\mathscr{V}(H^\sigma)}\|f\|_{L^{p}(\R)}\,.
    \] 
\end{theorem}

\subsection{Notation} 

The letter $C$ will be used throughout to denote various positive absolute constants that may change from line to line. 
The Fourier transform is 
\[
    \widehat{f}(\xi) = \int_{\R} e^{-2\pi i \xi x} f(x) \, dx\,.
\]
If $I \subset \R$ is an interval, we denote by $c(I)$ its center. 
We say that a smooth function $\varphi$ is adapted to an interval $I\subset \R$ if 
\begin{equation}\label{e:adapted}
    |\varphi^{(n)}(x)|\le |I|^{-n-1/2}\Big( 1+ \frac{|x-c(I)|}{|I|}\Big)^{-10}
\end{equation}
for $n = 0,1$. $M$ denotes the Hardy-Littlewood maximal function
\[
    Mf(x) = \sup_{s \in \mathbb{Z}} \frac{1}{\pi} \int_\R |f(x - 2^s y)| \frac{1}{1 + y^2}\, dy\,,
\]
and $M^p$ denotes the operator $M^pf = (M|f|^p)^{1/p}$. For a sequence $(\tau_s)_{s \in \mathbb{Z}}$, we denote by $M_\tau$ the $\tau$-shifted maximal function
\begin{equation}\label{e:M_tau}
    M_\tau f(x) = \sup_{s \in \mathbb{Z}} \frac{1}{\pi} \int_{\R} |f(x - 2^{s} (y - \tau_s))| \frac{1}{1 + y^2} \, dy\,.
\end{equation}
We also write $M_\tau^pf = (M_\tau |f|^p)^{1/p}.$
For $f\in L^p(\R)$ we   write
\[
    \|f\|_p = \|f\|_{L^p(\R)}.
\]
For $\lambda>0$ we denote the $L^1$ normalized dilation of a function $f$ by a factor $\lambda$ by
\[ 
    D_\lambda f(x) = \lambda^{-1} f(\lambda^{-1}x).
\]
Finally, we use the notation $\varphi\in C\mathcal{S}_0^\tau$ to say that $\varphi=C\widetilde{\varphi}$ for some $\widetilde{\varphi}\in \mathcal{S}_0^\tau$.

\subsection*{Acknowledgment.} 
L.~B. thanks Yu-Hsiang Lin and Martin Hsu for some discussions about the H\"ormander multiplier theorem. 
P.~D. was partially supported by the  Simons Foundation grant  MPS-TSM-00013943.

\section{Outline of the proof of boundedness of the shifted BHT}
\label{s:overview}

This section outlines the proof of Theorem \ref{thm:main2} by stating the key lemmas of the argument and establishing Theorem \ref{thm:main2} assuming these lemmas. The proofs of the lemmas themselves are given in the next section. 

Denote by $\Theta_0$ the set of all functions $\psi$ with frequency support in $[8,9]$  which satisfy
\[
    |\psi(x)| \le  (1 + |x|)^{-100}\,.
\]
Let further $\Theta_0^{\tau}$ be the set of $\tau$-shifted functions 
\[
    \Theta_0^{\tau} = \{T_{\tau} \psi \ : \ \psi \in \Theta_0\}\,.
\]
By a Littlewood-Paley decomposition of $\psi$ and a scaling argument, it suffices to prove Theorem \ref{thm:main2} with $\mathcal{S}^\tau_0$ replaced by  $\Theta_0^\tau$, and we will do so. We fix the parameter $m \ge 0$ and the sequence $(\tau_s)_{s \in \mathbb{Z}}$.

\subsection{Discretization}  The first step in the proof of Theorem~\ref{thm:main2} is a  discretization of the operator $B_{2^s}(\psi_s,f_1,f_2)$, tailored to the class of functions $\Theta_0^{\tau_s}$. 

A tile is a rectangle $p = I_p \times \omega_p$ such that $I_p, \omega_p$ are half-open intervals of the form $[u, v)$ and $|I_p||\omega_p| = 1$. We say that a function $\psi$ is adapted to the tile $p$  if its Fourier transform $\widehat \psi$ is supported in $\omega_p$ and the function
$x\mapsto e^{-2\pi i c(\omega) x} \psi(x)$ is adapted to the interval $I_p$ in the sense of \eqref{e:adapted}. The set of all functions adapted to $p$ is denoted by $\Psi(p)$.
A tri-tile $\mathbf{p}$ of scale $s = s(\mathbf{p})$ is a six-tuple 
\[
    \mathbf{p} = (p_1, p_2, p_3) = (I_{p_1}, \omega_{p_1}, I_{p_2}, \omega_{p_2}, I_{p_3},  \omega_{p_3})
\]
such that $p_i = (I_{p_i}, \omega_{p_i})$ is a tile for $i = 1,2,3$ and
\[
    |I_{p_1}| = |I_{p_2}| = |I_{p_3}| = 2^s\,.
\]

The following specific collections of tri-tiles will be used in the proof. We note that the details of the numerology here are not important at all, it is chosen to satisfy Lemma \ref{l:lac} and the properties listed below. 
For a quadruple of integers $\nu = (a, b, s',\alpha)$ with 
\[
    21\le a \le 30, \qquad -6\le b \le 3 , \qquad 0\le s' \le 9, \qquad 0\le \alpha\le 2  ,
\]
we define the collections of tri-tiles
\begin{equation*}
    {P}_{\nu} = \{ \mathbf{p}_{\nu}(s,v, \ell) \ : \ v, \ell, s \in \mathbb{Z}, \,s \equiv s' \!\!\!\pmod{10}, \,  \ell \equiv \alpha \!\!\!\pmod{3}\}
\end{equation*}
where 
\begin{align*}
    \mathbf{p}_\nu(s,v, \ell) &= \Big( 2^s[v - \tau_s, v+1  -\tau_s), \ 2^{-s}\Big[\frac{\ell}{3}, \frac{\ell}{3}+1\Big),\\
    &\quad \ \ 2^s[v + \tau_s, v+1  + \tau_s), \ 2^{-s}\Big[\frac{\ell - a }{3},  \frac{\ell - a}{3}+1\Big),\\
    &\quad \ \ 2^s[v, v+1), \ 2^{-s}\Big[\frac{2\ell - a - b}{3},  \frac{2\ell - a - b}{3}+1\Big)\Big)\,.
\end{align*}
Note that  the spatial intervals of the tri-tiles in $P_\nu$ of scale $s(\mathbf{p}_\nu) = s$ are shifted relative to each other, they satisfy
\begin{equation}\label{e:interval_shift}
     I_{p_1} = I_{p_3} - \tau_s|I_{p_3}|, \qquad\qquad I_{p_2} = I_{p_3} + \tau_s|I_{p_3}|\,.
\end{equation}
A grid is a set $\mathcal{D}$ of dyadic intervals with the property that for $\omega, \omega' \in \mathcal{D}$
\[
  \omega \cap \omega' \ne \emptyset \qquad \implies \qquad \omega \subset \omega' \quad \text{or} \quad \omega' \subset \omega\,.
\]
The frequency intervals of the tri-tiles in $P_\nu$ lie in one of three grids, depending on $\nu$. Indeed, for $\delta\in \{0,1,2\}$, denote by $\mathcal{D}_\delta$ the shifted dyadic grid  
\[
    \mathcal{D}_\delta = \{2^{-s}[n + (-1)^s\delta/3, n+(-1)^s\delta/3 + 1) \ : \  s,n\in \Z\}\,.
\]
Then for every $\nu=(a,b,s',\alpha)$ and $\epsilon=(-1)^{s'}$, it holds for $(p_1, p_2, p_3) = \mathbf{p} \in P_\nu$ that
\[
    (\omega_{p_1}, \omega_{p_2}, \omega_{p_3}) \in \mathcal{D}_{\epsilon \alpha \bmod 3} \times \mathcal{D}_{\epsilon (\alpha -a) \bmod 3} \times \mathcal{D}_{\epsilon (2\alpha -a-b) \bmod 3}\,.
\]
Moreover, if $\omega, \omega'$ are frequency intervals of tree-tiles in $P_\nu$, then we have the stronger separation of scales
\begin{equation}\label{e:sparse_scales}
    \omega \subsetneq \omega' \qquad \implies \qquad |\omega| \le 2^{-10} |\omega'|\,.
\end{equation}
The time intervals $I_{p_i}$ of the tri-tiles in $P_\nu$ all belong to the standard dyadic grid $\mathcal{D}_0$.

\subsection{The dyadic model operator}

The form $B_{2^s}(\psi, f_1, f_2)$ can be expanded into wave-packets  adapted to the tiles we just defined. 

\begin{lemma}
    \label{l:disc}
    There exists a constant $C > 0$ such that the following holds. 
    Let $s \in \mathbb{Z}$ and $\psi\in \Theta_0^{\tau_s}$. There exist wave-packets $\psi_{n, \mathbf{p},i} \in \Psi(p_i)$ such that
    \[
        B_{2^s}(\psi, f_1, f_2) = C\sum_{\nu} \sum_{n \in \Z^2}  (1 + |n|)^{-10} \sum_{\mathbf{p}  \in {P}_\nu \, : \, s(\mathbf{p}) = s} |I_{p_3}|^{-1/2} \langle f_1, \psi_{n, \mathbf{p},1} \rangle \langle f_2, \psi_{n, \mathbf{p},2} \rangle \overline{\psi_{n, \mathbf{p},3}}\,. 
    \]
\end{lemma}

To prove Theorem \ref{thm:main2}, by summing in $s$ it remains to show for any fixed $\nu$ and $n$ and any collection of wave-packets $\psi_{\mathbf{p},i} \in \Psi(p_i)$ 
\begin{equation}\label{e:goal}
    \Big\| \sum_{\mathbf{p}\in P_\nu} |I_{p_3}|^{-1/2} \langle f_1, \psi_{n, \mathbf{p},1} \rangle \langle f_2, \psi_{n, \mathbf{p},2} \rangle \overline{\psi_{n, \mathbf{p},3}} \Big\|_{p} \le C m^{4} \|f_1\|_{p_1} \|f_2\|_{p_2}\,.
\end{equation}
We fix $\nu$ and $n$ and suppress them from here on in the notation.  

\subsection{Interpolation} An interpolation and scaling argument further reduce \eqref{e:goal} 
to a restricted weak type estimate.  Let $\Lambda$ be the dualized model form 
\[
    \Lambda(f_1, f_2, f_3) = \sum_{\mathbf{p}} |I_{p_3}|^{-1/2} |\langle f_1, \psi_{ \mathbf{p},1} \rangle|| \langle f_2, \psi_{\mathbf{p},2} \rangle| |\langle f_3, \psi_{ \mathbf{p},3} \rangle|\,,
\]
where as before $\psi_{\mathbf{p}, i} \in \Psi(p_i)$ for all $i$ and $\mathbf{p}$, and the sum is over any finite set of tri-tiles. Then it suffices to prove the following.

\begin{proposition}\label{p:interpolation} 
    There exists $C> 0$ such that the following holds. 
    Let $E_1, E_2, E_3$ be measurable sets with 
    \[
        |E_{1}| \le |E_{2}| \le |E_{3}| \qquad\qquad \text{and} \qquad\qquad |E_3| \in (1,2]\,.
    \]
    Then there exists a subset $\widetilde E_3 \subset E_3$ with $2|\widetilde E_3| \ge |E_3|$ such that for every $\Lambda$ as above
    \begin{equation}\label{e:goal2}
        |\Lambda(\mathbf{1}_{E_1}, \mathbf{1}_{E_2}, \mathbf{1}_{\widetilde E_3})| \le C m^{4} (1+\lvert\log_2 \lvert E_1\rvert\rvert) |E_1| |E_2|^{1/2}  \,,
    \end{equation}
    and the same is true for any other permutation of the arguments of $\Lambda$. 
\end{proposition} 

\subsection{Trees and sizes}
The general strategy of the proof from here on is to efficiently organize the collection of all tri-tiles 
into sub-collections, and to estimate the contribution of each sub-collection in a manner that can be summed up. We now define combinatorial objects named trees, the sub-collections will be such trees.  

\begin{figure}
    \begin{tikzpicture}[scale=5.5]

    \definecolor{lightgreen}{HTML}{C8E6C9}
    \definecolor{mossgreen}{HTML}{8BC34A}
    \definecolor{palemint}{HTML}{A7FFEB}

\draw[fill=gray!20] (0,0) rectangle (2,0.1);

\draw[fill=gray!20] (0,-0.2) rectangle (1,-0.1);
\draw[fill=gray!20] (1,-0.2) rectangle (2,-0.1);

\foreach \i in {0,...,3} {
    \pgfmathsetmacro{\x}{\i/2}
    \pgfmathsetmacro{\xnext}{(\i+1)/2}
    \draw[fill=gray!20] (\x,-0.4) rectangle (\xnext,-0.3);
}

\foreach \i in {1} {
    \pgfmathsetmacro{\x}{\i/2}
    \pgfmathsetmacro{\xnext}{(\i+1)/2}
    \draw[fill=mossgreen!180] (\x,-0.4) rectangle (\xnext,-0.3);
}

\foreach \i in {0,...,7} {
    \pgfmathsetmacro{\x}{\i/4}
    \pgfmathsetmacro{\xnext}{(\i+1)/4}
    \draw[fill=gray!20] (\x,-0.6) rectangle (\xnext,-0.5);
}

\foreach \i in {6,7} {
    \pgfmathsetmacro{\x}{\i/4}
    \pgfmathsetmacro{\xnext}{(\i+1)/4}
    \draw[fill=mossgreen!80] (\x,-0.6) rectangle (\xnext,-0.5);
}

\foreach \i in {0,...,15} {
    \pgfmathsetmacro{\x}{\i/8}
    \pgfmathsetmacro{\xnext}{(\i+1)/8}
    \draw[fill=gray!20] (\x,-0.8) rectangle (\xnext,-0.7);
}

\foreach \i in {8,...,11} {
    \pgfmathsetmacro{\x}{\i/8}
    \pgfmathsetmacro{\xnext}{(\i+1)/8}
    \draw[fill=mossgreen!80] (\x,-0.8) rectangle (\xnext,-0.7);
}

\foreach \i in {0,...,31} {
    \pgfmathsetmacro{\x}{\i/16}
    \pgfmathsetmacro{\xnext}{(\i+1)/16}
    \draw[fill=gray!20] (\x,-1) rectangle (\xnext,-0.9);
}

\foreach \i in {12,...,19} {
    \pgfmathsetmacro{\x}{\i/16}
    \pgfmathsetmacro{\xnext}{(\i+1)/16}
    \draw[fill=mossgreen!80] (\x,-1) rectangle (\xnext,-0.9);
}

\foreach \i in {0,...,63} {
    \pgfmathsetmacro{\x}{\i/32}
    \pgfmathsetmacro{\xnext}{(\i+1)/32}
    \draw[fill=gray!20] (\x,-1.2) rectangle (\xnext,-1.1);
}

\foreach \i in {20,...,35} {
    \pgfmathsetmacro{\x}{\i/32}
    \pgfmathsetmacro{\xnext}{(\i+1)/32}
    \draw[fill=mossgreen!80] (\x,-1.2) rectangle (\xnext,-1.1);
}
\end{tikzpicture}
\caption{Green intervals $I_{p_2}$ of an $(i,3)$-tree with dark-green top interval, for $m = 2$.}
\end{figure}
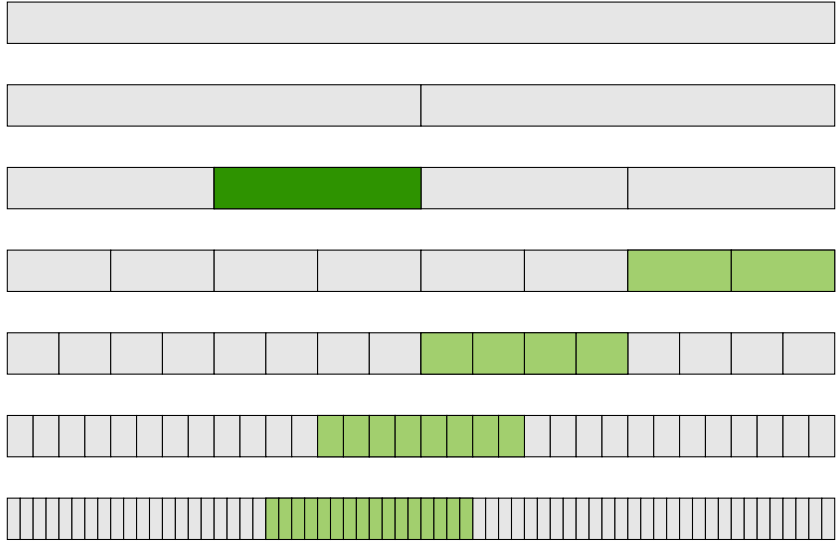

Let $i,j\in \{1,2,3\}$. An $(i,j)$-tree $T$ with top $I_T$ and central frequency $\xi_T$ is a collection of tri-tiles in $P$ such that for all $\mathbf{p} \in T$
\begin{equation}\label{e:tree}
    \xi_T \in 3\omega_{p_i}\,\qquad  \textup{and}\, \qquad I_{p_j} \subset I_{T}.
\end{equation}
This differs from the standard definition in the time-frequency literature in that there is usually no parameter $j$. It is necessary in our setting because of the shift \eqref{e:interval_shift}, which causes $I_{p_1}$, $I_{p_2}$ and $I_{p_3}$ to be different intervals. In the final steps of the proof $j$ will be choosen so that the $j$-the argument of $\Lambda$ is the indicator of the largest of the three sets. For example, as \eqref{e:goal2} is written, $j$ will equal $3$, and in general it will be the index of the argument $\mathbf{1}_{\widetilde E_3}$. This is to make the trees localized outside of the exceptional set $E_3 \setminus \widetilde E_3$ in the $j$-th component, and it is why we need the freedom to choose $j$.

The following is a standard lacunarity lemma. It will be used to obtain orthogonality of the wave packets in two of the three components of a tree.

\begin{lemma}
    \label{l:lac}
    Suppose that $T$ is an $(i,j)$-tree. Then for $i' \ne i$, there exists $\xi_T'$ such that for all tri-tiles $\mathbf{p} \in T$, 
    \begin{equation}
        \label{e:lacunary}
        \xi_T' \in 50 \omega_{p_{i'}} \setminus 2\omega_{p_{i'}}. 
    \end{equation}
\end{lemma}

For a collection of tri-tiles $P'$ and a function $f$, we define the $(i,j, k)$-size of $(f,P')$ by 
\[
    \mathrm{size}_{i,j,k}(f, P') = \sup_{\substack{T \subset P'\\ T \ \text{$(i,j)$-tree}}} \Big( \frac{1}{|I_T|} \sum_{\mathbf{p} \in T} |\langle f ,\psi_{\mathbf{p},k} \rangle|^2 \Big)^{1/2} 
\]
if $i \ne k$,  and else by
\[
    \mathrm{size}_{k,j,k}(f, P') = \sup_{\mathbf{p} \in P'} \Big( \frac{1}{|I_{p_k}|} |\langle f ,\psi_{\mathbf{p},k} \rangle|^2 \Big)^{1/2}\,.
\]
The indices $(i,j)$  indicate the type of tree, while the index $k$ specifies which of the three functions we take the size of.

The contribution of a single tree to the model form $\Lambda$ is bounded by a product of sizes.
\begin{lemma}
    \label{l:sizecs}
    For every $(i,j)$-tree $T$ and all $f_1, f_2, f_3 \in L^2(\R)$
    \[
        \Big|\sum_{\mathbf{p} \in T} |I_{p_3}|^{-1/2} \langle f_1, \psi_{\mathbf{p},1} \rangle \langle f_2, \psi_{\mathbf{p},2} \rangle \langle f_3, \psi_{\mathbf{p},3} \rangle\Big|\le |I_T| \prod_{k \in \{1,2,3\}} \size_{i,j,k}(f_k, T).
    \]
\end{lemma}
Indeed, by   H\"older's inequality, the left-hand side is bounded by
\begin{align*}
    |I_T| \Big(\sup_{\mathbf{p} \in T} |I_{p_i}|^{-1/2} |\langle f_i, \psi_{\mathbf{p},i} \rangle| \Big) \prod_{i' \ne i} \Big( \frac{1}{|I_T|} \sum_{\mathbf{p} \in T} |\langle f_{i'}, \psi_{\mathbf{p},i'}\rangle|^2 \Big)^{1/2},
\end{align*}
which is bounded by the product of sizes on the right hand side. 
 
\subsection{A weak Bessel inequality}
Keeping Lemma \ref{l:sizecs} in mind, we need control of both the quantities $\size(f_k, T)$ and $|I_T|$ for all trees that we use in the decomposition of the model form. The latter will be a consequence of the former, via a weak Bessel inequality. Such inequalities are also standard in time-frequency analysis arguments. A key difference here is that the shift \eqref{e:interval_shift} in the definition of a tri-tile generates more overlap in physical space and therefore less orthogonality of the trees. This causes the additional factor $m$ in Lemma \ref{l:bessel} below. 

Fix $i,j,k \in \{1,2,3\}$ with $i \ne k$.
We say that two $(i,j)$-trees are $k$-strongly disjoint if for all $\mathbf{p} \in T$, $\mathbf{p}' \in T'$ with $\omega_{p_k} \subsetneq \omega_{p_k'}$, we have that $I_{p_j'} \not\subset I_T$. Note that this definition depends on $i,j$, and $k$. We  establish the following weak Bessel inequality  with a linear loss in $m$.

\begin{lemma}\label{l:bessel}
    There exists $C> 0$ such that the following holds.
    Let $i,j,k \in \{1,2,3\}$ with  $i \ne k$.
    Assume that $\mathcal{T}$ is a collection of $(i,j)$-trees that are pairwise $k$-strongly disjoint. Assume further that for each tri-tile $\mathbf{p} \in \cup \mathcal{T}$, 
    \begin{equation}
        \label{tile_upper}
        |\langle f_k, \psi_{\mathbf{p},k}\rangle|^2 \le 4 \lambda^2 |I_{p_k}|
    \end{equation}
    and that for each $T \in \mathcal{T}$,  
    \begin{equation}
        \label{tree_lower}
              \sum_{\mathbf{p} \in T} |\langle f_k , \psi_{\mathbf{p},k} \rangle|^2 \ge \frac{\lambda^2}{2} |I_T|.
    \end{equation}
    Then 
    \begin{equation}\label{e:goalBessel}
        \sum_{T \in \mathcal{T}} |I_T| \le C m^{3/2}\frac{\|f_k\|_2^2}{\lambda^2}.
    \end{equation}
\end{lemma}
 
\subsection{The tree selection algorithm}

With the definitions and the weak Bessel inequality in place, we turn to the decomposition algorithm to split the model form into trees. It is based on an iterative application of the following lemma.

\begin{lemma}
\label{l:tree}
    There exists $C> 0$ such that the following holds.
    Let $j,k \in \{1,2,3\}$. 
    Let $P'$ be a finite collection of tri-tiles with 
    \begin{equation}\label{e:size_aspt}
        \mathrm{size}_{i,j,k}(f_k, P') \le 2\lambda
    \end{equation}
    for each $i  \in \{1,2,3\}$. Then there exist collections $\mathcal{T}_{i'}$ of $(i',j)$-trees,   $i' \in\{1,2,3\}$, such that for each $i  \in \{1,2,3\}$, 
    \begin{equation}
        \label{e:small_size}
        \mathrm{size}_{i,j,k} \Big(f_k,P' \setminus \bigcup_{i'=1}^3 \bigcup_{T \in \mathcal{T}_{i'}} T\Big) \le \lambda
    \end{equation}
    and 
    \begin{equation}
        \label{e:small_support}
        \sum_{T \in \bigcup_{i'=1}^3 \mathcal{T}_{i'}} |I_T| \le C m^{3/2} \frac{\|f_k\|_2^2}{\lambda^2}.
    \end{equation}
\end{lemma}

\subsection{Control of sizes} To start the iterative application of Lemma \ref{l:tree}, one needs some initial control of the sizes. This is accomplished by the following lemma, which controls them outside of exceptional sets. 

From here on we will assume that the arguments of $\Lambda$ are permuted as in \eqref{e:goal2}, that is, $|E_1| \le |E_2| \le |E_3|$, and we estimate for a major subset $\widetilde E_3 \subset E_3$ the form $\Lambda(\mathbf{1}_{E_1}, \mathbf{1}_{E_2}, \mathbf{1}_{\widetilde E_3})$. We then set $j = 3$. For other permutations of the arguments the proof is the same after making the corresponding changes, however the notation in a general argument would be cumbersome.

\begin{lemma}
    \label{l:refsize}
    There exist constants $C, C_0 > 0$ such that the following holds.
    Let $E_1, E_2, E_3$ be measurable sets in $\R$ such that $|E_1| \le |E_2| \le |E_3|$ and $1/2< |E_3| \le 1$.
    Set
    \[
        q_1 = 1, \qquad\qquad q_2 = 2, \qquad\qquad q_3 = 2\,.
    \]
    Let $M^q$ and $M_\tau^q$ be the maximal functions defined in \eqref{e:M_tau}. Define the exceptional set 
    \begin{equation}\label{e:F}
        F = \bigcup_{k =1,2} \{ M^{q_k} \mathbf{1}_{E_k} > C_0 |E_k|^{1/{q_k}}\} \cup \bigcup_{k =1,2} \bigcup_{c = -2}^2 \{M^{q_k}_{c\tau} \mathbf{1}_{E_k} > C_0 (m |E_k|)^{1/q_k}\}.
    \end{equation}
    Then $F$ is open, $|F| < 1/12$ and the following holds. Let $P'\subset P$ be any collection of tri-tiles with $I_{p_3} \not\subset F$ for all $\mathbf{p} \in P'$. Then it holds for all $i, k\in \{1,2,3\}$ that
    \begin{equation}\label{e:size_est}
        \size_{i,3,k}(|E_k|^{-1/q_k}\mathbf{1}_{E_k}, P') \le 
        \begin{cases}
        Cm^{3/2} &\text{if $k =1$,}\\
        Cm &\text{if $k = 2$,}\\
        C &\text{if $k=3$.}
        \end{cases}
    \end{equation}
\end{lemma}

\subsection{Completing the proof}

We now prove Proposition \ref{p:interpolation}, and hence Theorem \ref{thm:main2}.

\begin{proof}[Proof of Proposition \ref{p:interpolation}]

The argument that follows does not use anything that is not preserved under permutation of the functions. 
We therefore prove only the  case of \eqref{e:goal2} as stated;  the remaining cases, corresponding to other permutations of the arguments of $\Lambda$, follow by a similar reasoning.

Let $F$ be the exceptional set defined in Lemma \ref{l:refsize}. Let $\widetilde F$ be the union of $3J$, where $J$ runs through all the maximal dyadic intervals contained in $F$. Note that $F$ is the union of the intervals $J$, since $F$ is open. We pick the set $\widetilde{E}_3$ as
\[
    \widetilde{E}_3 = E_3 \setminus \widetilde{F}\,. 
\]
Note that $|\widetilde F| \le 3|F|\le 1/4$ and $|E_3| > 1/2$, so that $2|\widetilde E_3| \ge |E_3|$ as required. Denoting 
\[
    g_1 = \mathbf{1}_{E_1} |E_1|^{-1/2}, \qquad g_2 = \mathbf{1}_{E_2} |E_2|^{-1/2}, \qquad g_3 = \mathbf{1}_{\widetilde E_3} |E_3|^{-1/2}\,,
\]
the remaining claim \eqref{e:goal2} of Proposition \ref{p:interpolation} becomes that for any finite set of tri-tiles $P' \subset P$
\[
    \sum_{\mathbf{p}\in  P'} |I_{p_3}|^{-1/2} |\langle g_1, \psi_{\mathbf{p}, 1} \rangle \langle g_2, \psi_{\mathbf{p}, 2} \rangle \langle g_3, \psi_{\mathbf{p}, 3} \rangle|  \le Cm^{4} |E_1|^{1/2}(1+\lvert\log_2 \lvert E_1\rvert^{1/2}\rvert)\,.
\]
We split the sum over $P'$ as 
\begin{equation}
    \label{e:splitsum}
        \sum_{\mathbf{p} \in P'}   = \sum_{\mathbf{p} \in P' : I_{p_3} \subset F}   + \sum_{\mathbf{p} \in P' : I_{p_3} \not\subset F} \,
\end{equation}
and treat the two sums on the right-hand side separately.

\textbf{Tiles in the exceptional set $F$:} We denote by $\mathcal{I}$ the collection of maximal dyadic intervals $J$ contained in $F$. They are pairwise disjoint and cover $F$.
For every tri-tile $\mathbf{p}$ with $I_{p_3} \subset F$,  there exists some $J \in \mathcal{I}$ with $I_{p_3} \subset J$, and there is some $d \ge 0$ with $|I_{p_3}| = 2^{-d}|J|$. 
We order the tiles by  such maximal intervals $J$ and then  by $d$ and $I_{p_3}$
\begin{equation}\label{e:inside_ex}
    \sum_{\substack{\mathbf{p} \in P' \\ I_{p_3} \subset F}} |I_{p_3}|^{-1/2} \prod_{k=1}^3|\langle g_k, \psi_{\mathbf{p}, k} \rangle|  
     = \sum_{J \in \mathcal{I}} \sum_{d \ge 0} \sum_{\substack{I \subset J\\ |I| = 2^{-d} |J|}} \sum_{\substack{\mathbf{p}\in P'\\  I_{p_3} = I}}|I|^{-1/2}   \prod_{k=1}^3 |\langle g_k, \psi_{\mathbf{p},k} \rangle|.
\end{equation}
We estimate the inner sum in $\mathbf{p}$ using H\"older's inequality by
\begin{equation}\label{e:CS}
    \sum_{\mathbf{p}:  I_{p_3} = I}|I|^{-1/2}  \prod_{k =1}^3 |\langle g_k, \psi_{\mathbf{p},k} \rangle| \le  \Big(  \prod_{k=1}^2 \Big( \sum_{\mathbf{p}:  I_{p_3} = I}|\langle g_k, \psi_{\mathbf{p},k} \rangle|^2 \Big )^{1/2} \Big) \sup_{\mathbf{p}: I_{p_3}=I}|I|^{-1/2}|\langle g_3, \psi_{\mathbf{p},3}\rangle | .
\end{equation}
For fixed $I$ and $k$, the wave packets $\psi_{\mathbf{p},k}$ with $I_{p_3} = I$ are pairwise orthogonal. Since $\psi_{\mathbf{p},k}$ are rapidly decaying away from $I_{p_k}$, we obtain  for $k=1,2$ 
\begin{equation*}
     \sum_{\mathbf{p}:  I_{p_3} = I}|\langle g_k, \psi_{\mathbf{p},k} \rangle|^2  \le  C\int_{\R} |g_k(x)|^2 \Big(1+ \frac{|x-c(I_{p_k})|}{|I|}\Big)^{-2} dx \,.
\end{equation*}
Since $J \subset F$ is maximal, there exists some $x_0 \in \widehat{J} \setminus F$, where $\widehat{J}$ denotes the parent dyadic interval of $J$. Since $|I| = 2^{-d}|J|$, this $x_0$ has distance at most $C 2^{d} |I|$ from $I$. Thus, the previous is bounded by 
\[
    \le C 2^{2d} \int_\R |g_k(x)|^2  \Big(1+ \frac{|x-x_0 - c\tau_{s(\mathbf{p})}|I||}{|I|}  \Big)^{-2} \, dx \le C 2^{2d} |I| (M^2_{c\tau} g_k(x_0))^2\,.
\]
Using the definition of $F$ in Lemma \ref{l:refsize}, we conclude that
\begin{equation}\label{e:first_fact}
        \sum_{\mathbf{p}:  I_{p_3} = I}|\langle g_k, \psi_{\mathbf{p},k} \rangle|^2 \le C 2^{2d} m^{2/q_k} |I|    |E_k|^{2/q_k - 1} = \begin{cases}
        C 2^{2d} m^2 |I|    |E_1| & \text{if $k=1$,}\\
        C 2^{2d} m |I|     & \text{if $k = 2$.}
    \end{cases}
\end{equation}
For the last factor of \eqref{e:CS}, we have by definition of $g_3$ and $|E_3| \ge 1/2$ that
\begin{equation}\label{e:sec_fact}
    |I|^{-1/2} |\langle g_3, \psi_{\mathbf{p},3}\rangle|  \le C |I|^{-1} \int_{\widetilde E_3} \Big(1 + \frac{|x - c(I)|}{|I|}\Big)^{-10} \, dx \le C 2^{-9d}\,.
\end{equation}
The final inequality hold since $I\subset J$ and $3J$ is disjoint from $\widetilde E_3$, so that for $x \in \widetilde E_3$
\[
    |x-c(I)|\ge |J| = 2^d|I|\,.
\] 
Combining \eqref{e:first_fact} and \eqref{e:sec_fact}, we obtain for \eqref{e:CS} the bound
\[
    \sum_{\mathbf{p}:  I_{p_3} = I}|I|^{-1/2}  \prod_{k =1}^3 |\langle g_k, \psi_{\mathbf{p},k} \rangle| \le C2^{-7d}m^{3/2} |I|  |E_1|^{1/2}\,.
\]
Summing in $I$ and using that the intervals $J \in \mathcal{I}$ are pairwise disjoint and contained in $F$ and $|F| \le 1/12$, it finally follows that 
\[
    \sum_{\mathbf{p}:I_{p_3} \subset F} |I_{p_3}|^{-1/2} \prod_{k=1}^3|\langle g_k, \psi_{\mathbf{p}, k} \rangle|  \le Cm^{3/2} |E_1|^{1/2} \sum_{J \in \mathcal{I}} |J| \le Cm^{3/2} |E_1|^{1/2}\,.
\]

\textbf{Tiles outside of the exceptional set $F$.}
We turn to the second sum on the right-hand side of \eqref{e:splitsum}. Let $P''=\{\mathbf{p}: I_{p_j}\not\subset F\}$, then by Lemma \ref{l:refsize} we have for every $i\in \{1,2,3\}$ that
\begin{align}    
    \size_{i,3,1}(m^{-3/2}g_1, P'') &\le C\,,\qquad \qquad \size_{i,3,2}(m^{-1}g_2, P'') \le C\,,\nonumber\\
    &\size_{i,3,3}(g_3, P'') \le C\,. \label{sizeest_g_k}
\end{align}
Applying Lemma \ref{l:tree} iteratively, we obtain a decomposition
\[
    P'' = \bigcup_{\ell \ge 0} \bigcup_{i=1}^3 \bigcup_{T \in \mathcal{T}_{\ell, i}} T,
\]
where for each $i\in \{1,2,3\}$, the trees $T \in \mathcal{T}_{\ell,i}$ are $(i, 3)$-trees that satisfy 
\begin{align}    
    \size_{i,3,1}(m^{-3/2}g_1, T) &\le C2^{-\ell}\,,\qquad \qquad \size_{i,3,2}(m^{-1}g_2, T) \le C2^{-\ell}\,,\nonumber\\
    &\size_{i,3,3}(g_3, T) \le C2^{-\ell}\,,\label{sizeest_g_k1}
\end{align}
and such that 
\begin{equation}
    \label{rootest}
        \sum_{T\in \mathcal{T}_{\ell,i}} |I_T| \le C m^{3/2} 2^{2\ell}.
    \end{equation}
We rewrite and estimate the sum over $P''$ using Lemma \ref{l:sizecs} as 
\begin{equation}
    \label{sizecs}
    \sum_{\ell\ge0} \sum_{i=1}^3 \sum_{T \in \mathcal{T}_{\ell,i}} \sum_{\mathbf{p} \in T} |I_{p_3}|^{-1/2} \prod_{k=1}^3\lvert \langle g_k, \psi_{\mathbf{p}, k} \rangle\rvert  \le \sum_{\ell\ge0} \sum_{i=1}^3 \sum_{T \in \mathcal{T}_{\ell,i}} |I_{T}|\prod_{k=1}^3 \textup{size}_{i,3,k}(g_k, T)\,.
\end{equation}
For the portion of the sum in  \eqref{sizecs} where $\ell > \lvert \log_2 \lvert E_1\rvert^{1/2}\rvert$, we use \eqref{sizeest_g_k1},  and \eqref{rootest} to bound
\[
    \sum_{\ell>\lvert\log_2 \lvert E_1\rvert^{1/2}\rvert} \sum_{i=1}^3 \sum_{T \in \mathcal{T}_{\ell,i}} |I_{T}|\prod_{k=1}^3 \textup{size}_{i,3,k}(g_k, T) \le C m^4 \sum_{\ell>\lvert\log_2 \lvert E_1\rvert^{1/2}\rvert} 2^{-\ell} \le  C m^4  |E_1|^{1/2}. 
\]
For the portion where $\ell \le \lvert\log_2 \lvert E_1 \rvert^{1/2}\rvert$, we use  \eqref{sizeest_g_k1} for $k=2,3$, and Lemma \ref{l:refsize} for $k = 1$ which states that
\[
    \size_{i,3,1}(m^{-3/2}g_1, P'') \le C|E_1|^{1/2}\,.  
\]
We obtain, using again \eqref{rootest}, 
\[
    \sum_{\ell\le\lvert\log_2 \lvert E_1\rvert^{1/2}\rvert} \sum_{i=1}^3 \sum_{T \in \mathcal{T}_{\ell,i}} |I_{T}|\prod_{k=1}^3 \textup{size}_{i,3,k}(g_k, T)  \le C m^4|E_1|^{1/2}(1+\lvert\log_2 \lvert E_1\rvert^{1/2}\rvert)\,. 
\]
This completes the proof. 
\end{proof}

\section{Proofs of the intermediate lemmas}
\label{s:BHT}
In this section,  we prove  the lemmas stated in Section \ref{s:overview}.   
\subsection{Discretization}
\begin{proof}[Proof of Lemma \ref{l:disc}]
    By changing variables $y\to 2^s y$ and rescaling of the functions $f_j$, we may assume that $s = 0$. 
    
    Let $\rho$ be a Schwartz function with the property that $\widehat \rho$ is supported in $[0.1,0.9]$ and
    \[
        \sum_{\ell \in \mathbb{Z}} \Big|\widehat \rho\Big(\xi - \frac{\ell}{3}\Big)\Big|^2 = 1
    \]
    for all $\xi\in \R$.   Since $\rho$ is Schwartz, there exists a constant $C$ such that 
    \begin{equation}\label{e:psiadapted}
        |\rho(x)| \le C(1 + |x|)^{-100}\,.
    \end{equation}
    The function
    \[
        \xi\mapsto  \overline{\widehat{\rho}\Big(\xi - \frac{\ell}{3}\Big)} \widehat f(\xi)
    \]
    is supported in
    \[
         \Big[\frac{\ell}{3}+0.1, \frac{\ell}{3}+0.9  \Big] \subset  \Big[\frac{\ell}{3}, \frac{\ell}{3}+1 \Big] = \omega_\ell.
    \]
    Expanding it into its Fourier series on $\omega_\ell$ gives for $\xi \in \omega_\ell$
    \[
        \overline{\widehat{\rho}\Big(\xi - \frac{\ell}{3}\Big)} \widehat f(\xi) = \sum_{v\in \Z} e^{-2\pi i v \xi}\int_{\omega_\ell} \overline{\widehat{\rho}\Big(\eta - \frac{\ell}{3}\Big)} \widehat f(\eta)e^{2\pi i v \eta}\, d\eta.
    \]
    Thus, for all Schwartz functions $f$, 
    \begin{align}
        \widehat f(\xi) &= \sum_{\ell \in \mathbb{Z}}  \widehat \rho\Big(\xi - \frac{\ell}{3}\Big) \overline{\widehat \rho\Big(\xi - \frac{\ell}{3}\Big)} \widehat f(\xi)\nonumber\\
        &= \sum_{\ell \in \mathbb{Z}} \sum_{v \in \mathbb{Z}} \widehat \rho\Big(\xi - \frac{\ell}{3}\Big) e^{-2\pi i v \xi} \int_{\omega_\ell} \overline{\widehat \rho\Big(\eta - \frac{\ell}{3}\Big)} e^{2\pi i v \eta} \widehat f(\eta) d\eta\nonumber\\
        &= \sum_{\ell \in \mathbb{Z}} \sum_{v \in \mathbb{Z}} \widehat \psi_{v,\ell}(\xi) \langle f, \psi_{v,\ell} \rangle, \label{e:decomposition}
    \end{align}
    where 
    \[
        \psi_{v,\ell}(x) = e^{2\pi i\ell(x-v)/3} \rho(x - v)\,.
    \]
    We use the decomposition \eqref{e:decomposition} on $f_1, f_2$ and $B_{1}(\psi,f_1, f_2)$ to write 
    \begin{equation}\label{e:sum}
        B_1(\psi, f_1, f_2)= 
       \sum_{\substack{\ell_1, \ell_2,\ell_3 \in \mathbb{Z}\\ v_1, v_2, v_3 \in \mathbb{Z}}}  \langle f_1, \psi_{v_1,\ell_1}\rangle \langle f_2, \psi_{v_2,\ell_2} \rangle c_{v_1,v_2,v_3,\ell_1,\ell_2,\ell_3} \overline{\psi_{v_3, \ell_3}},
    \end{equation}
    where
    \[
        c_{v_1,v_2,v_3,\ell_1,\ell_2,\ell_3} = \big\langle B_1(\psi, \psi_{v_1, \ell_1}, \psi_{v_2, \ell_2}), \overline{\psi_{v_3, \ell_3}} \big\rangle\,.
    \]
     By Fourier inversion, 
    \begin{align}
       c_{v_1,v_2,v_3,\ell_1,\ell_2,\ell_3} = \iint_{\R^2} \widehat{\psi}_{v_1, \ell_1}(\xi_1) \widehat{\psi}_{v_2, \ell_2}(\xi_2) \widehat{\psi}_{v_3, \ell_3}(\xi_1 + \xi_2) \widehat{\psi}(\xi_1 - \xi_2) \, d\xi_1 \, d\xi_2\,. \label{e:0}
    \end{align}
    Hence, $c_{v_1,v_2,v_3,\ell_1,\ell_2,\ell_3}$ vanishes unless
    \[
        \Big[\frac{\ell_1}{3}, \frac{\ell_1}{3}+1\Big] - \Big[\frac{\ell_2}{3}, \frac{\ell_2}{3}+1\Big] = \Big[\frac{\ell_1 - \ell_2}{3} - 1, \frac{\ell_1 - \ell_2}{3} + 1\Big] 
    \]
    intersects the Fourier support $[8,9]$ of $\widehat {\psi}$. 
    Thus, it vanishes unless $a = \ell_1 - \ell_2$ satisfies
    \begin{equation}\label{e:vanish1}
        21 \le a \le 30\,.
    \end{equation}
    The coefficient $c_{v_1,v_2,v_3,\ell_1,\ell_2,\ell_3}$ in  \eqref{e:0} also vanishes unless
    \[
        \Big[\frac{\ell_1}{3}, \frac{\ell_1}{3}+1\Big] + \Big[\frac{\ell_2}{3}, \frac{\ell_2}{3}+1\Big] = \Big[\frac{\ell_1 + \ell_2}{3}, \frac{\ell_1 + \ell_2}{3}+2\Big] 
    \]
    intersects the Fourier support $[\ell_3/3, \ell_3/3+1]$ of $\psi_{v_3,\ell_3}$, so unless $b = \ell_1 + \ell_2 - \ell_3$ satisfies
    \begin{equation}\label{e:vanish2}
         - 6 \le b \le 3\,.
    \end{equation}
    We continue to take care of spatial localization. Let $n=(n_1,n_2)$ be given by  
    \begin{equation}\label{e:def_n}
        n_1=v_1 - v_3 + \tau_0,\qquad\qquad n_2=v_2 - v_3 - \tau_0.
    \end{equation} 
    Making the change of variables $x\to x+v_3$, $y\to y+\tau_0$ yields
    \begin{align*}
        c_{v_1,v_2,v_3,\ell_1,\ell_2,\ell_3}
        &= \iint_{\R^2} \psi_{v_3,\ell_3}(x) \psi_{v_1,\ell_1}(x - y) \psi_{v_2,\ell_2}(x+y) \psi(y) \, dy \, dx \\
        &= \iint_{\R^2} \psi_{0,\ell_3}(x) \psi_{0,\ell_1}(x - y-n_1) \psi_{0,\ell_2}(x+y-n_2) \psi(y+\tau_0) \, dy \, dx\,.
    \end{align*}
    Using the decay \eqref{e:psiadapted} of $\rho$ and that $\psi \in \Theta_0^{\tau_0}$ bounds the previous by a constant times 
    \[
        \iint_{\R^2} (1+|x|)^{-100}(1+|x - y-n_1|)^{-100} (1+|x + y-n_2|)^{-100}  (1+|y|)^{-100}   \, \mathrm{d}y \, \mathrm{d}x\,.
    \]
    Computing the integral, we conclude
    \begin{equation}
        \label{decay}
       |c_{v_1,v_2,v_3,\ell_1,\ell_2,\ell_3}| \le    C(1 + |n|)^{-99}\,.
    \end{equation}
    Reorganizing the sum \eqref{e:sum} using the vanishing conditions \eqref{e:vanish1}, \eqref{e:vanish2} and the change fo variables \eqref{e:def_n} yields
    \[ 
        B_1(\psi, f_1, f_2) = \sum_{a=21}^{30} \sum_{b = -6}^3 \sum_{n\in \Z^2}  \sum_{\substack{v, \ell\in \mathbb{Z}}} 
        \langle f_1, \psi_{v-\tau_0+n_1,\ell}\rangle \langle f_2, \psi_{v+\tau_0+n_2,\ell-a} \rangle  c_{n_1,n_2,v, \ell, a, b }' \overline{\psi_{v, 2\ell-a-b}}\,.
    \]
    Here we abbreviated $c_{n_1,n_2,v, \ell,a,b}'=  c_{v-\tau_0+n_1,v+\tau_0+n_2,v,\ell,\ell-a,2\ell-a-b}$. For each $n\in \Z^2$ and  $\mathbf{p}= (p_1,p_2,p_3)=\mathbf{p}_{\nu}(0,v,\ell)$ with $\nu=(a,b,0,\alpha)$, we set
    \begin{align*}
        \psi_{n,\mathbf{p},1} &= (1 + |n|)^{-10} \psi_{v-\tau_0+n_1,\ell}\,,\\
        \psi_{n,\mathbf{p},2} &= (1 + |n|)^{-10} \psi_{v+\tau_0+n_2,\ell-a}\,,\\
        \psi_{n,\mathbf{p},3} &= (1 + |n|)^{30} c_{n_1,n_2,v, \ell, a, b }' \psi_{v, 2\ell-a-b}\,. 
    \end{align*}
    The wave-packet $\psi_{n,\mathbf{p},1}$ is adapted to $p_{1}$,  thanks to the   factor $(1 + |n|)^{-10}$. Similarly,  $\psi_{n,\mathbf{p},2}$ is adapted to $p_{2}$, while  $\psi_{n,\mathbf{p},3}$ is adapted to $p_{3}$ due to \eqref{decay}. 
\end{proof}

\subsection{Interpolation}\label{ss:interpolation_proof} 
Suppose that Proposition \ref{p:interpolation} holds.

\begin{lemma}\label{l:gen_rest_type}
    There exists $C> 0$ such that the following holds.
    Let $E_1, E_2, E_3 \subset \R$ be measurable and let   $(\alpha_1, \alpha_2, \alpha_3)$ be in the interior of the convex hull $A$ of the $6$ permutations of 
    \[
        (1, 1/2, -1/2)\,.
    \]
    Let $E_j$ have the largest measure of the three sets. Then there exists $\widetilde E_j \subset E_j$ with $2|\widetilde E_j| \ge |E_j|$ and, writing $\widetilde E_i = E_i$ for $i \ne j$,
    \begin{equation}\label{e:int1}
        |\Lambda(\mathbf{1}_{\widetilde E_1}, \mathbf{1}_{\widetilde E_2}, \mathbf{1}_{\widetilde E_3})| \le Cm^4 |E_1|^{\alpha_1}|E_2|^{\alpha_2}|E_3|^{\alpha_3}\,.
    \end{equation}
\end{lemma}

\begin{proof}
    Without loss of generality $|E_1| \le |E_2| \le |E_3|$. It suffices to prove \eqref{e:int1} in a small neighborhood of the corners of $A$. By the rearrangement inequality, it suffices to prove it in a small neighborhood of $(1, 1/2, -1/2)$, where it follows from \eqref{e:goal2} and scaling. 
\end{proof}

The conclusion of Lemma \ref{l:gen_rest_type} is called a restricted weak type estimate for $\Lambda$. It implies the claimed $L^p$ estimates in the range \eqref{e:p_aspt}, as is discussed in detail for example in Section 3 of \cite{MuscaluTT2002}.

\subsection{Trees and sizes}

\begin{proof}[Proof of Lemma \ref{l:lac}]
    Recall that for any tri-tile $\mathbf{p}_\nu(s,v, \ell)$, there exist $21 \le a \le 30$ and $-6 \le b \le 3$ with 
    \[
        (\omega_{p_1}, \omega_{p_2}, \omega_{p_3}) = 2^{-s}\Big(\Big[\frac{\ell}{3}, \frac{\ell}{3}+1\Big), \, \Big[\frac{\ell - a }{3},  \frac{\ell - a}{3}+1\Big), \, \Big[\frac{2\ell - a - b}{3},  \frac{2\ell - a - b}{3}+1\Big)\Big)\,.
    \]   
    We will prove that for any $i \ne i'$ it holds that if $\xi_T \in \omega_i$ then $\xi_T' \in 50\omega_{i'} \setminus 2\omega_{i'}$, where we set
    \[
        \xi_T' = \begin{cases}
            \xi_T \qquad &\text{if $i' = 1,2$,}\\
            2\xi_T \qquad &\text{if $i' = 3$.}
        \end{cases}
    \]
    This claim is invariant under changing $s$ or $\ell$. Hence we can assume that $s = 0$ and we are free to fix $\ell$. As further preparation, note that if $\omega = [0,1) - c$ then 
    \begin{equation}\label{e:prep1}
        [-1, 2) \subset 50 \omega \setminus 2\omega \quad \iff \quad c \in \Big[-\frac{47}{2} , -\frac{5}{2}\Big] \cup \Big[\frac{5}{2}, \frac{47}{2}\Big]
    \end{equation}
    and 
    \begin{equation}\label{e:prep2}
        [-2, 4) \subset 50 \omega \setminus 2\omega \quad \iff \quad c \in \Big[-\frac{45}{2} , -\frac{9}{2}\Big] \cup \Big[\frac{7}{2}, \frac{43}{2}\Big]\,. 
    \end{equation}
    
    Let first $i = 1$. We may fix $\ell = 0$, so that our assumption is $\xi_T \in 3\omega_{p_1} =  [-1, 2)$, or, equivalently,  $2\xi_T \in [-2, 4)$. Since $\ell = 0$, we further have
    \[
        \omega_{p_2} = [0,1) - \frac{a}{3}, \qquad\qquad \omega_{p_3} = [0,1) - \frac{a+b}{3}\,,
    \]
    where $a/3 \in [7,10]$ and $(a+b)/3 \in [5,11]$. Using \eqref{e:prep1} and \eqref{e:prep2}, the claim follows. 
    
    When $i = 2$ we choose $\ell = a$. Then the assumption is again that $\xi_T \in 3\omega_{p_2} = [-1, 2)$ and $2\xi_T \in [-2, 4)$, and we have 
    \[
        \omega_{p_1} = [0,1) + \frac{a}{3}, \qquad\qquad \omega_{p_3} = [0,1) + \frac{a-b}{3},
    \]
    where $a/3 \in [7,10]$ and $(a-b)/3 \in [6,12]$. By \eqref{e:prep1} and \eqref{e:prep2}, this implies the claim.

    Finally, when $i = 3$ we choose $\ell = (a + b)/2$, so that $\omega_{p_3} = [0,1)$ and the assumption is $\xi_T \in 3\omega_{p_3} = [-1,2)$. Now
    \[
        \omega_{p_1} = [0,1) + \frac{a+b}{6}, \qquad\qquad \omega_{p_2} = [0,1) + \frac{b-a}{6}\,,
    \]
    with $(a+b)/6 \in [5/2, 11/2]$ and $(b-a)/6 \in [-6,-3]$. This completes the proof by \eqref{e:prep1}. 
\end{proof}

\subsection{A weak Bessel inequality}

The proof of the Bessel inequality differs from the argument without shift in the estimate \eqref{e:mclaim}, which we prove separately in Lemma \ref{l:claim} below. This estimate is proved by splitting into essentially two cases: If two tiles have very different scales, then the shift at the smaller scale is negligible compared to the size of the larger tile. In that case a minor variant of the non-shifted argument applies. There remain boundedly many scale-differences where this does not work, which can be treated by simpler arguments. 

\begin{proof}[Proof of Lemma \ref{l:bessel}]
    We will write $f$ for $f_k$ and we will write $p = (I, \omega)$ and $p' = (I', \omega')$ for the components $p_k$ and $p_k'$ of tri-tiles $\mathbf{p}$ and $\mathbf{p}'$, respectively. We will also denote $\psi_{p}=\psi_{\mathbf{p},k},\; \psi_{p'}=\psi_{\mathbf{p}',k}$, and assume by scaling that $\|f\|_2=1$.  
  
    It suffices to prove the lemma under the additional assumption that for all trees $T \in \mathcal{T}$ 
    \begin{equation}
        \label{tree_upper}
            \sum_{\mathbf{p} \in T} |\langle f , \psi_{p}\rangle|^2 \le 2\lambda^2 |I_T|.
    \end{equation}
    Indeed, we may otherwise split $\mathcal{T}$ as the disjoint union of the sets, for $n\ge 0$, 
    \[
        \mathcal{T}_n = \{T\in \mathcal{T} : 2^{2n-1} \lambda^2 \le |I_T|^{-1}\sum_{\mathbf{p} \in T} |\langle f , \psi_{p}\rangle|^2 < 2^{2n+1}\lambda^2\}\,.
    \]
    Note that the trees in $\mathcal{T}_n$  satisfy
    \eqref{tree_lower} and \eqref{tree_upper} with $\lambda$ replaced by $2^{n}\lambda$. Thus, once we have proved the lemma with the stronger assumption \eqref{tree_upper}, it follows that 
    \[
        \sum_{T \in \mathcal{T}_n} |I_T| \le C m^{3/2}\frac{\|f_k\|_2^2}{2^{2n}\lambda^2}.
    \]
    Summing over all $n\ge 0$ then yields the lemma without the assumption \eqref{tree_upper}.

    From now on we assume \eqref{tree_upper} holds. By \eqref{tree_lower} and the Cauchy-Schwarz inequality,
    \begin{align}
        \sum_{T \in \mathcal{T}} |I_T| &\le 2\lambda^{-2} \sum_{\mathbf{p} \in \cup \mathcal{T}} |\langle f, \psi_{p}\rangle|^2= 2\lambda^{-2} \Big\langle f, \sum_{\mathbf{p} \in \cup \mathcal{T}} \langle f, \psi_p \rangle \psi_p \Big \rangle\nonumber\\
        &\le 2\lambda^{-2} \Big( \sum_{\mathbf{p} \in \cup \mathcal{T}} \sum_{\mathbf{p}' \in \cup \mathcal{T}} |\langle f, \psi_{p} \rangle| |\langle f, \psi_{p'}\rangle| |\langle \psi_{p}, \psi_{p'} \rangle| \Big)^{1/2}\,.\label{e:start}
    \end{align}
    First we estimate the contribution of tri-tiles $\mathbf{p}$ and $\mathbf{p}'$ of equal scale to the sum in \eqref{e:start}. By the Cauchy-Schwarz inequality, 
    \begin{equation}
        \sum_{\mathbf{p} \in \cup \mathcal{T}} \sum_{\substack{\mathbf{p'}\in \cup \mathcal{T}\\|I'| = |I|}} |\langle f, \psi_{p} \rangle ||\langle f, \psi_{p'}\rangle| |\langle \psi_p, \psi_{p'} \rangle| \le \sum_{\mathbf{p} \in \cup \mathcal{T}} |\langle f, \psi_p \rangle|^2\sum_{\substack{\mathbf{p'}\in \cup \mathcal{T}\\|I'| = |I|}}  |\langle \psi_p, \psi_{p'} \rangle|\,.\label{e:equal_scale}
    \end{equation}
    If a tri-tile $\mathbf{p}'$ contributes to the sum, then $\omega' = \omega$, as $\psi_p$ and $\psi_{p'}$ are otherwise orthogonal.
    Such tri-tiles are uniquely determined by the   interval $I'$. If $I' = \ell|I| + I$, then
    \[
        |\langle \psi_p, \psi_{p'} \rangle| \le C(1 + |\ell|)^{-10}. 
    \]
    Using this and \eqref{tree_upper}, the term \eqref{e:equal_scale} is controlled by  
    \begin{equation}
        \label{e:same_scale}
         C\sum_{\mathbf{p} \in \cup \mathcal{T}} |\langle f, \psi_p \rangle|^2\sum_{\ell \in \mathbb{Z}}  (1 + |\ell|)^{-10} \le C \sum_{\mathbf{p} \in \cup \mathcal{T}} |\langle f, \psi_p \rangle|^2 \le C \lambda^2 \sum_{T \in \mathcal{T}} |I_T|.
    \end{equation}
    This completes our estimate with  $\mathbf{p}, \mathbf{p}'$ of equal scale in \eqref{e:start}. 
    
    We turn to the contribution of tiles of distinct scale in  \eqref{e:start}. Using the  Cauchy-Schwarz inequality, \eqref{tile_upper}, and \eqref{tree_upper}, one obtains
    \begin{align*}
         \sum_{\mathbf{p} \in \cup \mathcal{T}} \sum_{\substack{\mathbf{p'}\in \cup \mathcal{T}\\|I'| < |I|}} & |\langle f, \psi_{p} \rangle ||\langle f, \psi_{p'}\rangle| |\langle \psi_p, \psi_{p'} \rangle|\\
        &\le \sum_{T \in \mathcal{T}} \Big( \sum_{\mathbf{p} \in T} |\langle f, \psi_p \rangle|^2 \Big)^{1/2} \Big( \sum_{\mathbf{p} \in T} \Big(\sum_{\substack{\mathbf{p'}\in \cup \mathcal{T}\\|I'| < |I|}} |\langle f, \psi_{p'} \rangle ||\langle \psi_p, \psi_{p'}\rangle|\Big)^2 \Big)^{1/2}\\
        &\le C\lambda^2 \sum_{T \in  \mathcal{T}} |I_T|^{1/2} \Big( \sum_{\mathbf{p} \in T} \Big(\sum_{\substack{\mathbf{p'}\in \cup \mathcal{T}\\|I'| < |I|}} |I'|^{1/2} |\langle \psi_p, \psi_{p'}\rangle|\Big)^2 \Big)^{1/2}.
    \end{align*}
    We will prove in Lemma \ref{l:claim} below that for each tree $T \in \mathcal{T}$,  
    \begin{equation}
        \label{e:mclaim}
        \sum_{\mathbf{p} \in T} \Big(\sum_{\substack{\mathbf{p'}\in \cup \mathcal{T}\\|I'| < |I|}} |I'|^{1/2} |\langle \psi_p, \psi_{p'}\rangle|\Big)^2 \le C m^3|I_T|\,.
    \end{equation}
    Assuming this holds, we obtain from \eqref{e:start}, \eqref{e:same_scale}, and \eqref{e:mclaim} that 
    \[
        \sum_{T \in \mathcal{T}} |I_T| \le \lambda^{-2} \Big( C \lambda^2 \sum_{T \in \mathcal{T}} |I_T| + Cm^{3/2} \lambda^2 \sum_{T \in \mathcal{T}} |I_T|\Big)^{1/2},
    \]
    which completes the proof of \eqref{e:goalBessel} upon rearranging and recalling that $\|f\|_2 = 1$.
\end{proof}

\begin{lemma}\label{l:claim}
    Let $i,j, k \in \{1,2,3\}$ with $i \ne k$ and let $\mathcal{T}$ be a collection of $(i,j)$-trees that are pairwise $k$-strongly disjoint. Then \eqref{e:mclaim} holds.
\end{lemma}

\begin{proof}
    We use the same conventions as in the proof of the previous lemma in that we write $p = (I, \omega)$ and $p' = (I', \omega')$ for the components $p_k$ and $p_k'$ of tri-tiles $\mathbf{p}$ and $\mathbf{p}'$.

    The key ingredient here that is not available when $|I| = |I'|$ is that tiles $\mathbf{p}$ and   $\mathbf{p'}$ contributing to \eqref{e:mclaim} necessarily belong to different trees. Indeed, suppose that on the contrary they both belong to the same tree $T$.  
    Since $i\neq k$, by Lemma \ref{l:lac} it holds for some $\xi_T'$
    \begin{equation}\label{e:both_lac}
        \xi_T'\in 50\omega \setminus 2\omega\,, \qquad\qquad\qquad
        \xi_T'\in 50\omega' \setminus 2\omega'\,.
    \end{equation}
    Furthermore, since $|\langle \psi_p, \psi_{p'}\rangle| \ne 0$, it must hold that $\omega\cap \omega'\ne \emptyset$.  Because $|I'|<|I|$, it follows that $\omega\subset \omega'$. By sparseness of the scales \eqref{e:sparse_scales} this implies the stronger $2^{10}|\omega|\le  |\omega'|$. It follows that $50\omega \subset 2\omega'$, contradicting \eqref{e:both_lac}. 

    By strong disjointness of the trees, it follows in particular that $I_{p_j}'$ is disjoint from $I_T$ for each $\mathbf{p}'$ contributing in \eqref{e:mclaim}.

   Similarly, if $\mathbf{p}' \ne \mathbf{p}''$ both contribute in the inner sum in \eqref{e:mclaim}, then $I_{p_j'} \cap I_{p_j''} = \emptyset$. Suppose first that they are in different trees. Then since $\omega \subset \omega'$ and $\omega \subset \omega''$ it holds that $\omega' \subset \omega''$ or $\omega'' \subset \omega'$. Since $\mathbf{p} \ne \mathbf{p}''$, the inclusion is strict or $I_{p_j'} \cap I_{p_j''} = \emptyset$. If the inclusion is strict, then $I_{p_j'} \cap I_{p_j''} = \emptyset$ follows from strong disjointness of the trees containing $\mathbf{p}'$ and $\mathbf{p}''$. Now suppose that they are in the same tree $T'$. Again it follows without loss of generality that $\omega' \subsetneq \omega''$ and hence, by \eqref{e:sparse_scales} that $50\omega \subset \omega'$. This contradicts Lemma \ref{l:lac} for the tree $T'$.

    We split up the sum in \eqref{e:mclaim} according to the relative sizes of $I$ and $I'$ and of $I$ and $I_T$. We will use parameters $d, e$ defined by 
    \[
        |I| = 2^{-d}|I_T|, \qquad\qquad |I'| = 2^{-e} |I|, \qquad\qquad d \ge 0,\qquad\qquad e \ge 10\,.
    \]
    By the Cauchy-Schwarz inequality, 
    \begin{align*}
        \eqref{e:mclaim} &\le \sum_{\mathbf{p} \in T} \Big(\sum_{\substack{\mathbf{p'}\in \cup \mathcal{T}\\|I'| < 2^{-m-1}|I|}} |I'|^{1/2} |\langle \psi_p, \psi_{p'}\rangle|\Big)^2\\
        &+ m \sum_{e = 10}^{m+1} \sum_{\substack{\mathbf{p} \in T\\|I| < 2^{-10m}|I_T|}} \Big(\sum_{\substack{\mathbf{p'}\in \cup \mathcal{T}\\|I'| = 2^{-e}|I|}} |I'|^{1/2} |\langle \psi_p, \psi_{p'}\rangle|\Big)^2\\
        &+ m \sum_{d = 0}^{10m}\sum_{e = 10}^{m+1} \sum_{\substack{\mathbf{p} \in T\\|I| = 2^{-d}|I_T|}} \Big(\sum_{\substack{\mathbf{p'}\in \cup \mathcal{T}\\|I'| = 2^{-e}|I|}} |I'|^{1/2} |\langle \psi_p, \psi_{p'}\rangle|\Big)^2\\
        &= T_1 + T_2 + T_3\,.
    \end{align*}
    \textbf{The term $T_1$:}   
    A direct computation using \eqref{e:adapted} yields that
    \begin{align}\label{e:single_pair0}  
         |\langle \psi_p, \psi_{p'}\rangle| \le C |I|^{-1/2}|I'|^{1/2} \Big(1 + \frac{|c(I) - c(I')|}{|I|}\Big)^{-10}\,. 
    \end{align}
    By \eqref{e:interval_shift} we have $|c(I') - c(I_{p_j'})| \le 2\tau_{s(\mathbf{p}')}|I'| \le 2^{m+1}|I'|$. Using further that $e \ge m+2$ one obtains
    \begin{align}
        &\le C |I|^{-1/2}|I_{p_j'}|^{1/2} \Big(1 - \frac{2^{m+1}|I'|}{|I|} + \frac{|c(I) - c(I_{p_j'})|}{|I|}\Big)^{-10}\nonumber\\\label{e:pair_with_shift}
        &\le C |I|^{-1/2}|I_{p_j'}|^{1/2} \Big(1 + \frac{|c(I) - c(I_{p_j'})|}{|I|}\Big)^{-10}\,.
    \end{align}
    As proved above, the intervals $I_{p_j'}$ are pairwise disjoint and disjoint from $I_T$.
    Combining this with the estimate \eqref{e:pair_with_shift} and $|I_{p_j'}| < |I|$ yields
    \begin{align*}
        \sum_{\substack{\mathbf{p'}\in \cup \mathcal{T}\\|I'| < 2^{-m-1}|I|}} |I'|^{1/2} |\langle \psi_p, \psi_{p'}\rangle| &\le C |I|^{-1/2} \int_{(I_T)^c} \Big(1 + \frac{|c(I) - x|}{|I|}\Big)^{-10} \, dx\\
        &\le C |I|^{1/2} \Big(1 + \frac{\mathrm{dist}(I, (I_T)^c)}{|I|}\Big)^{-9}.
    \end{align*}
    Summation in $\mathbf{p}$ yields
    \begin{equation}
        \label{case1}
        T_1 \le C \sum_{\mathbf{p} \in T} |I| \Big(1 + \frac{\mathrm{dist}(I, (I_T)^c)}{|I|}\Big)^{-18}.
    \end{equation}
    As the summand here depends only on $I$ and not on the other data associated with  $\mathbf{p}$, we rewrite this in terms of the interval only. 
    Let $\mathcal{I}_d(I_T)$ denote the collection of all dyadic intervals $J \subset I_T$ with $|J| = 2^{-d} |I_T|$.
    By the definition \eqref{e:tree} of a tree, for every $J \subset I_T$, there are at most $3$ tri-tiles $\mathbf{p} \in T$ with $I_{p_j} = J$. By the shift property \eqref{e:interval_shift} of the intervals in a tri-tile, the corresponding $I = I_{p_k}$ is of the form $I =  J + c 2^m |J|$ for some $c \in [-2,2]$. Hence we may write \eqref{case1} as 
    \[
        T_1 \le C\sum_{d \ge 0} \sum_{J \in \mathcal{I}_d(I_T)} |J| \, w_{J}\,,  \qquad\qquad w_{J} = \Big(1 + \frac{\dist(J + c2^m|J|, (I_T)^c)}{|J|}\Big)^{-18}\,.
    \]
    We split up the sum as
    \[
        S_1 + S_2 + S_3 = \sum_{d = 0}^{10m-1} \sum_{J \in \mathcal{I}_d(I_T)} + \sum_{d = 10m}^\infty \sum_{\substack{J \in \mathcal{I}_d(I_T)\\ \dist (J, (I_T)^c)\le 2^{m+2}|J|}}  + \sum_{d = 10m}^\infty \sum_{\substack{J \in \mathcal{I}_d(I_T)\\ \dist (J, (I_T)^c)> 2^{m+2}|J|}} \,.
    \]
    Estimating $w_J$ by $1$, we find that $S_1 \le 10 m |I_T|$. Estimating $w_J$ by $1$ in $S_2$ gives
    \[ 
        S_2 \le C \sum_{d = 10m}^\infty 2^m 2^{-d} |I_T| \le C|I_T|\,.
    \] 
    In $S_3$ we use that, because of the smallness of $J$, 
    \[
        \dist(J + c2^m|J|, (I_T)^c)\ge \dist(J, (I_T)^c) - |c|2^m|J| \ge \frac{1}{2}\dist(J, (I_T)^c)\,.
    \]
    This yields
    \[
        \sum_{\substack{J \in \mathcal{I}_d(I_T)\\ \dist (J, (I_T)^c)> 2^{m+2}|J|}} |J| \Big(1 + \frac{\dist(J , (I_T)^c)}{|J|}\Big)^{-18} \le C2^{-d}|I_T| \sum_{\ell>2^{m+2}}  (1 + \ell)^{-18} \le C 2^{-d}|I_T| \,.
    \]
    Summing in $d$ we find $S_3 \le C|I_T|$. Combining the estimates for $S_1, S_2$ and  $S_3$ we conclude that 
    \[
        T_1 \le Cm |I_T|\,.
    \]

    \noindent\textbf{The term $T_3$:} 
    We fix first the parameters $d$ and $e$ in the outer two sums.
    This determines the scale of $\mathbf{p}'$. If $|\langle \psi_p, \psi_{p'} \rangle| \ne 0$, then $\omega \subset \omega'$, so this also determines $\omega'$. It follows that the intervals $I'$ corresponding to distinct such tri-tiles $\mathbf{p}'$ are pairwise disjoint. Combining this with  \eqref{e:single_pair0} gives
    \begin{equation}
        \label{e:int2}
        \sum_{\substack{\mathbf{p'}\in \cup \mathcal{T}\\|I'| = 2^{-e}|I|}} |I'|^{1/2} |\langle \psi_p, \psi_{p'}\rangle| \le C |I|^{-1/2} \int_{\R}\Big(1 + \frac{|c(I) - x|}{|I|}\Big)^{-10} dx \le C|I|^{1/2}.
    \end{equation}
    As in the estimate of $T_1$, it follows that  
    \[
        \sum_{\substack{\mathbf{p} \in T\\ |I| = 2^{-d}|I_T|}} \Big(\sum_{\substack{\mathbf{p'}\in \cup \mathcal{T}\\|I'| = 2^{-e}|I|}} |I'|^{1/2} |\langle \psi_p, \psi_{p'}\rangle|\Big)^2  \le C\sum_{J \in \mathcal{I}_d(I_T)} |J| = C|I_T|\,.
    \]
    Summing in $d$ and $e$, it follows that $T_3 \le C m^3 |I_T|$.\\

    \noindent\textbf{The term $T_2$.} We fix the parameter $e$ in the outer sum. If a tile $\mathbf{p}'$ contributes to the inner sum in \eqref{e:mclaim}, then by strong disjointness the interval  $I_{p_j}'$ is disjoint from $I_T$.  
    Since   $I' = I_{p_j'} + c2^m |I'|$  for some $c \in [-2,2]$ and $|I_{p_j'}| \le |I_{p_j}| = 2^{-d}|I_T|$, we obtain 
    \[
        I' \cap (1 - 2^{2 + m-d}) I_T = \emptyset. 
    \]
    Since the scale of $\mathbf{p}'$ is determined by $e$, we find as for term $T_3$ that the intervals $I'$ corresponding to the distinct $\mathbf{p}'$ are pairwise disjoint. Combining this with \eqref{e:single_pair0} we find 
    \begin{align*}
        \sum_{\substack{\mathbf{p'}\in \cup \mathcal{T}\\|I'| = 2^{-e}|I|}} |I'|^{1/2} |\langle \psi_p, \psi_{p'}\rangle| &\le C |I|^{-1/2} \int_{((1 - 2^{2+m-d})I_T)^c} \Big(1 + \frac{|c(I) - x|}{|I|}\Big)^{-10}\, dx\\
        &\le C |I|^{1/2} \Big(1 + \frac{\mathrm{dist}(I, ((1 - 2^{1-d + m})I_T)^c)}{|I|}\Big)^{-9}\,.
    \end{align*}
    Summing in $\mathbf{p}$ with $|I| < 2^{-10m}|I_T|$ and parameterizing the sum by $I$ as  in the estimate of $T_1$ bounds the contribution of the fixed value of $e$ to $T_2$ by
    \[
        C\sum_{d > 10m} \sum_{J \in \mathcal{I}_d(I_T)} |J|  \Big(1 + \frac{\mathrm{dist}(J, ((1 - 2^{1-d+ m})I_T)^c)}{|J|}\Big)^{-18}\,.
    \]
    A very similar analysis as for $S_2$ and $S_3$ in the estimate of $T_1$ bounds this by $C|I_T|$. Summing finally in $e$ yields that $T_2 \le Cm^2|I_T|$.
\end{proof}

\subsection{The tree selection algorithm}

Let $j,k\in\{1,2,3\}$ be given. For each $(i', j)$-tree $T$ with $i' \in \{1,2,3\} \setminus \{k\}$, Lemma \ref{l:lac} gives  a frequency $\xi_T'$ such that for all $\mathbf{p} \in T$,  
\[
    \xi_T' \in 50 \omega_{p_k} \setminus 2 \omega_{p_k}.
\]
Fixing a choice of $\xi_T'$, each such tree $T$ can thus be decomposed into the set of all tri-tiles $\mathbf{p}$ with $\xi_T' < c(\omega_{p_k})$ and the set of all tri-tiles $\mathbf{p}$ with $\xi_T' > c(\omega_{p_k})$, where $c(\omega_{p_k})$ denotes the center of $\omega_{p_k}$. We will call such collections $<$-trees and $>$-trees, respectively. Note that they are still $(i', j)$-trees.

\begin{proof}[Proof of Lemma \ref{l:tree}]
    We will construct the collections $\mathcal{T}_{i'}$, $i'\in \{1,2,3\}$. For $i'\ne k$, $\mathcal{T}_{i'}$ will be defined as the union of two collections $\mathcal{T}_{i'}^{>}$ and  $\mathcal{T}_{i'}^{<}$, while we only construct one collection $\mathcal{T}_k$. 
    
    We start the selection algorithm with each of these $5$ collections being empty and with $P_0 = P'$ and $\ell = 0$. We denote $\{1,2,3\} \setminus \{k\} = \{i_1', i_2'\}$.   

    \textbf{Step 1.} If there is no  $(i'_1, j)$-tree $T$ in $P_\ell$ that is a $>$-tree     
    and satisfies
    \begin{equation}
        \label{e:sel_crit}
        \sum_{\mathbf{p} \in T} |\langle f_k, \psi_{\mathbf{p},k} \rangle|^2 > \frac{\lambda^2}{2} |I_T|,
    \end{equation}
    then move to Step 2. Else, select such a tree $T$ which is maximal with respect to set inclusion and which has the minimal $\xi_T'$ among all such trees. Add $T$ to $\mathcal{T}_{i'_1}^{>}$. Further, add to the collection $\mathcal{T}_k$  the inclusion maximal $(k,j)$-tree $S$ in $P_\ell \setminus T$ with top interval $I_T$ and central frequency $\xi_T'$. Define $P_{\ell+1} = P_\ell \setminus (T \cup S) $, increase $\ell$ by $1$ and return to the start of Step $1$. 
    
    Since $P$ is finite, after finitely many iterations of Step 1, we end up with a set of tri-tiles containing no $(i'_1,j)$-tree which is a $>$-tree and satisfies \eqref{e:sel_crit}, and we move to Step 2. 

    \textbf{Step 2.} If there is no  $(i'_1, j)$-tree $T$ in $P_\ell$ that is a $<$-tree     
    and satisfies \eqref{e:sel_crit}
    then move to Step 3. Else, select such a tree $T$ which is maximal with respect to set inclusion and which has the maximal $\xi_T'$ among all such trees. Add $T$ to $\mathcal{T}_{i'_1}^{<}$. Further, add to the collection $\mathcal{T}_k$  the inclusion maximal $(k,j)$-tree $S$ in $P_\ell \setminus T$ with top interval $I_T$ and central frequency $\xi_T'$. Define $P_{\ell+1} = P_\ell \setminus (T \cup S)$, increase $\ell$ by $1$ and return to the start of Step $2$. 
    
    Again, after finitely many iterations we end up with a set of tri-tiles containing no $(i_1', j)$-tree which is a $>$-tree or a $<$-tree and which satisfies \eqref{e:sel_crit}.
 
    \textbf{Step 3.} Repeat this with the remaining value of $i'_2 \ne k$ in place of $i_1'$ to obtain collections     $\mathcal{T}_{i'_2} = \mathcal{T}_{i'_2}^{<} \cup \mathcal{T}_{i'_2}^{>}$ and possibly further adding to $\mathcal{T}_k$.

    Having constructed $\mathcal{T}_{i'}$ for $i' = 1,2,3$, it remains to verify \eqref{e:small_size} and \eqref{e:small_support}.  
    We start with \eqref{e:small_size}. For $i'\ne k$, every $(i', j)$-tree in the collection of remaining tiles with $i' \ne k$ is a union of a $>$-tree and a $<$-tree. Since neither satisfies \eqref{e:sel_crit}, it holds that
    \[
        \sum_{\mathbf{p} \in T} |\langle f_k, \psi_{\mathbf{p},k} \rangle|^2 \le  \lambda^2 |I_T| \,.
    \]
    Taking a supremum in $T$, estimate \eqref{e:small_size} follows.
    For $i' = k$, the size on the left-hand side of \eqref{e:small_size} is controlled by either of the $i'\ne k$ sizes, because a single tile is an $(i,j)$-tree for any $i,j$. Hence \eqref{e:small_size} also holds when $i' = k$. 

     To verify \eqref{e:small_support}, we will apply the weak Bessel inequality from Lemma \ref{l:bessel} to the collections $\mathcal{T}_{i'}^{>}$ and $\mathcal{T}_{i'}^{<}$, $i' \ne k$. This is sufficient, as the $(k,j)$-trees in $\mathcal{T}_k$ have the same top intervals as the $(i', j)$-trees selected at the same stage. The assumption \eqref{tile_upper} of Lemma \ref{l:bessel} holds by \eqref{e:size_aspt}, while the assumption \eqref{tree_lower} holds by the selection criterion \eqref{e:sel_crit} for all selected trees. It hence only remains to verify that the trees in $\mathcal{T}_{i'}^{>}$ and $\mathcal{T}_{i'}^{<}$ are $k$-strongly disjoint.
     We will show this for   $\mathcal{T}_{i'}^{>}$, the argument for $\mathcal{T}_{i'}^{<}$ is analogous. 

      Suppose that $\mathbf{p} \in T$ and $\mathbf{p'} \in T'$ with $T, T' \in \mathcal{T}_{i'}^{>}$ and that 
    \[
        \omega_{p_k} \subsetneq \omega_{p_k'}.
    \]
    Recall that 
    \[
        \xi_T' \in 50 \omega_{p_k} \setminus 2 \omega_{p_k}, \qquad \xi_{T'}' \in 50 \omega_{p_k'} \setminus 2 \omega_{p_k'}.
    \]
    Since the scales are $10$-separated and $\omega_{p_k} \subsetneq \omega_{p_k'}$, we have $50 \omega_{p_k} \subset 2 \omega_{p_k'}$. It follows that $\xi_{T'}' > \xi_{T}'$. So $T'$ was selected after $T$. But also $\xi_T' \in 50 \omega_{p_k} \subset 2 \omega_{p_k'}$. So if we had $I_{p_j'} \subset I_T$, then $\mathbf{p}'$ would have been contained in the tree in $\mathcal{T}_k$ selected right after $T$. Since it is not contained in any tree in $\mathcal{T}_k$, the strong disjointness property holds.
\end{proof}

\subsection{Control of sizes}

To prepare the proof of Lemma \ref{l:refsize}, we recall some maximal function and square function estimates, and show how they control the different sizes involved in our argument. 

To control $(i,j,k)$-sizes for $i = k$, we will use the following variant of the shifted maximal function. Let $\tau = (\tau_s)_{s \in \mathbb{Z}}$ be the sequence of shifts in Theorem \ref{thm:main2}, so that $2^{m-1} \le |\tau_s| \le 2^m$ for all $s$. Then we set
\[
    M_\tau f(x) = \sup_{s \in \mathbb{Z}} \int_{\R} |f(x' - 2^{s} (y - \tau_s))| \frac{1}{1 + y^2} \, dy\,.
\]
The following is a variant of the classical shifted maximal function estimate, see for example \cite{MS2013}. The difference here is that the shift $\tau_s$ depends mildly on $s$, however the proof in \cite{MS2013} is easily modified to cover that case as well. 

\begin{lemma}\label{l:max}
    The maximal function $M_\tau$ satisfies  
    \begin{equation}\label{e:weak2M}
        \|M_\tau f\|_{L^{1,\infty}} \le C \max\{1, m\} \|f\|_1.
    \end{equation}
\end{lemma}

Lemma \ref{l:max} allows us to bound $(k,j,k)$-sizes.

\begin{lemma}\label{l:iequalsk}
    Suppose that $i = k$. Then for some $c \in \{-2,-1,0,1,2\}$ depending on $k, j$, for every collection $P'$ of tri-tiles
    \[
        \size_{i,j,k}(f, P') \le \sup_{\mathbf{p} \in P'} \inf_{x \in I_{p_j}} M_{c\tau} f(x)\,.
    \]
\end{lemma}

\begin{proof}
    By adaptedness of the wave-packets, one has 
    \[
        |I_{p_k}|^{-1/2} |\psi_{\mathbf{p}, k}(x)|\le C|I_{p_k}|^{-1}\Big(1+\frac{|x-c(I_{p_k})|}{|I_{p_k}|}\Big)^{-2}.
    \] 
    Recalling from \eqref{e:interval_shift} that $I_{p_k} = I_{p_j} + c \tau_{s(\mathbf{p})} |I_{p_k}|$ for some $c \in \{-2,-1,0,1,2\}$, it follows for any $\mathbf{p} \in P$ that
    \begin{equation}\label{e:packetmax}
         |I_{p_k}|^{-1/2} |\langle f, \psi_{\mathbf{p}, k} \rangle| \le C \inf_{x \in I_{p_j}} M_{c\tau} f(x).
    \end{equation}
    Taking a supremum in $\mathbf{p} \in P'$, the lemma follows. 
\end{proof}

When $i \ne k$, control of the sizes requires an orthogonality argument.  

\begin{lemma}
    \label{l:orhogonality}
    Suppose that $T$ is an $(i,j)$-tree and $i \ne k$. Then for every function $f$,
    \[
        \sum_{\mathbf{p} \in T} |\langle f, \psi_{\mathbf{p}, k} \rangle|^2 \le C \|f\|_2^2\,.
    \]
\end{lemma}

\begin{proof}
    Without loss of generality we assume $\|f\|_2 = 1$. Then, by Cauchy-Schwarz
    \[
        \sum_{\mathbf{p} \in T} |\langle f, \psi_{\mathbf{p}, k} \rangle|^2 \le \Big( \sum_{\mathbf{p} \in T} \sum_{\mathbf{p}' \in T} |\langle f, \psi_{\mathbf{p}, k} \rangle||\langle f, \psi_{\mathbf{p}', k} \rangle||\langle  \psi_{\mathbf{p}, k}, \psi_{\mathbf{p}', k} \rangle| \Big)^{1/2}\,.
    \]
    By Lemma \ref{l:lac} and the separation of scales \eqref{e:sparse_scales}, it holds that $\langle \psi_{\mathbf{p}, k}, \psi_{\mathbf{p}', k}\rangle = 0$ unless $s(\mathbf{p}) = s(\mathbf{p}')$. Using also Cauchy-Schwarz once more, the previous is bounded by 
    \[
        \Big(\sum_{\mathbf{p} \in T} |\langle f, \psi_{\mathbf{p}, k}\rangle|^2 \sum_{\substack{\mathbf{p}' \in T\\ s(\mathbf{p}) = s(\mathbf{p}')}} |\langle  \psi_{\mathbf{p}, k}, \psi_{\mathbf{p}', k} \rangle| \Big)^{1/2} \le C \Big(\sum_{\mathbf{p} \in T} |\langle f, \psi_{\mathbf{p}, k} \rangle|^2\Big)^{1/2}\,,
    \]
    where the final inequality follows from the computation after \eqref{e:start}.
\end{proof}

\begin{lemma}\label{l:jneqk}
    Suppose that $i \ne k$. Then for every $j$ there exists $c \in \{-2,-1,0,1,2\}$ such that for every $(i, j)$-tree
    \begin{equation}\label{e:square1}
        \frac{1}{|I_T|} \sum_{\mathbf{p}\in T} |\langle f, \psi_{\mathbf{p},k} \rangle|^2 \le C\Big( m \sup_{\mathbf{p} \in T} \inf_{x \in I_{p_j}} (M_{c\tau} f(x))^2 + \frac{1}{|I_T|} \sum_{\mathbf{p} \in T, |I_{p_k}|\le 2^{-m}|I_T|} |\langle f, \psi_{\mathbf{p},k} \rangle|^2\Big)
    \end{equation}
    and hence for every collection $P'$ of tri-tiles
    \begin{equation}\label{e:square2}
        \size_{i,j,k}(f,P') \le C \sup_{\mathbf{p} \in P'} \inf_{x \in I_{p_j}} ( m^{1/2}  M_{c\tau} f(x) + M^2 f(x))\,.
    \end{equation}
    If $k = j$, then also
    \begin{equation}\label{e:square3}
        \size_{i,k,k}(f,P') \le C \sup_{\mathbf{p} \in P'} \inf_{x \in I_{p_j}}  M^2 f(x)\,.
    \end{equation}
\end{lemma}

\begin{proof}
    Recall from \eqref{e:packetmax} that for a suitable $c$ for each $\mathbf{p}$ 
    \[
        |\langle f, \psi_{\mathbf{p},k}\rangle|^2 \le |I_{p_j}| \inf_{x \in I_{p_j}} (M_{c\tau} f(x))^2\,.
    \]
    For each $J \subset I_T$ with $|J| = 2^{-d}|I_T|$, there are at most $3$ tri-tiles $\mathbf{p} \in T$ with $I_{p_j} = J$, and such $J$ are pairwise disjoint. Summing in $J$, we obtain 
    \begin{equation}\label{e:large scales}
        \frac{1}{|I_T|} \sum_{\substack{\mathbf{p} \in T\\ |I_{p_j}| = 2^{-d}|I_T|}} |\langle f, \psi_{\mathbf{p},k}\rangle|^2 \le C \sup_{\mathbf{p} \in T} \inf_{x \in I_{p_j}} (M_{c\tau} f(x))^2\,.
    \end{equation}
    This implies \eqref{e:square1}. 

    To deduce \eqref{e:square2} from \eqref{e:square1}, 
    we start by applying Lemma \ref{l:orhogonality} to $f \mathbf{1}_{8I_T}$
    \[
        \sum_{\substack{\mathbf{p}\in T\\ |I_{p_k}| \le 2^{-m}|I_T|}} |\langle f \mathbf{1}_{8I_T}, \psi_{\mathbf{p},k} \rangle|^2 \le C \|f \mathbf{1}_{8I_T}\|^2_2 \le C |I_T| \inf_{x \in I_T} (M^2 f(x))^2\,.
    \]
    It remains to take care of the contribution of $f \mathbf{1}_{(8 I_T)^c}$. Using Cauchy-Schwarz and the rapid decay of $\psi_p$ away from $I_{p_k}$, we estimate  
    \begin{align*}
        |\langle f \mathbf{1}_{( 8I_T)^c} , \psi_{\mathbf{p},k} \rangle |^2 &\le C\Big(\int_\R |f(x)| |I_{p_k}|^{-1/2} \Big(1+ \frac{|x-c(I_{p_k})|}{|I_{p_k}|}\Big)^{-10} \, dx\Big)^2\\
        &\le C \int_{\R} |f(x)|^2 \Big(1+ \frac{|x -c(I_T)|}{|I_T|}\Big)^{-2} \, dx\\
        &\quad \times \int_{(8I_T)^c} |I_{p_k}|^{-1} \Big(1+ \frac{|x-c(I_{p_k})|}{|I_{p_k}|}\Big)^{-20} \Big(1+ \frac{|x -c(I_T)|}{|I_T|}\Big)^{2}\, dx\,.
    \end{align*}
    Since $|I_{p_k}| \le 2^{-m}|I_T|$ and $I_{p_k} = I_{p_j} + c\tau_{s(\mathbf{p})}|I_{p_k}|$ and $I_{p_j} \subset I_T$, we have $I_{p_k} \subset 6 I_T$. It follows that the previous is bounded by 
    \[
        C  |I_T|\inf_{x\in I_T} (M^2f(x))^2 \times 
        \int_{(8I_T)^c} |I_{p_k}|^{-1} \Big(1+ \frac{|x-c(I_{p_k})|}{|I_{p_k}|}\Big)^{-18}\,dx\,.
    \]
    Suppose that $|I_{p_k}| = 2^{-d}|I_T|$ with $d \ge m$.
    Computing the integral using that $|I_{p_k}| \subset 6I_T$
    \[ 
         \le C 2^{-17d}   |I_T|\inf_{x\in I_T} (M^2f(x))^2  \, .
    \]
    Summing in $I_{p_k}$ and in $d$, we conclude
    \begin{align*}
        \sum_{d \ge m} \sum_{\substack{\mathbf{p} \in T\\ |I_{p_k}| = 2^{-d}|I_T|}} |\langle f\mathbf{1}_{\R \setminus 8I_T}, \psi_p\rangle|^2&\le C \sum_{d \ge m} 2^{-17d} \sum_{\substack{\mathbf{p} \in T\\ |I_{p_k}| = 2^{-d}|I_T|}} |I_T| \inf_{x \in I_T} (M^2 f(x))^2\\
        &= C \sum_{d \ge m} 2^{-16d} |I_T| \inf_{x \in I_T} (M^2 f(x))^2
        \le C |I_T| \inf_{x \in I_T} (M^2 f(x))^2\,.
    \end{align*}
    This concludes the proof of \eqref{e:square2}.
    In the $k = j$ case \eqref{e:square3} there is no shift, so that the above argument already applies to all tiles $\mathbf{p} \in T$ and not just those with $|I_{p_k}| \le 2^{-m}|I_T|$.
\end{proof}

To prove $L^p$ bounds for $p < 2$, the estimate \eqref{e:square2} is not quite good enough. There we use the following refined square function estimate from  \cite[Proposition 2.4.1]{T2006}.
\begin{proposition}[{\cite[Proposition 2.4.1]{T2006}}]\label{p:ref}
    Let $f \in L^1(\R)$ and  $\lambda>0$. Let  $F_\lambda$ be the union of all dyadic intervals $J$ such that 
    \[
        \int_{J} |f| \ge \lambda |J|.
    \]
    Then for each dyadic interval $K$, and any mean zero functions $\psi_I$ adapted to dyadic intervals $I$
    \[
        \sum_{I \subset K, I \not\subset F_\lambda} |\langle f, \psi_I \rangle|^2 \le \lambda^2 |K|\,.
    \]
\end{proposition}

\begin{proof}[Proof of Lemma \ref{l:refsize}]
    By weak $L^{q_k}$ boundedness of the Hardy-Littlewood maximal function
    \[
        |\{ M^{q_k} \mathbf{1}_{E_k} > C_0 |E_k|^{1/{q_k}}\}| \le \frac{C}{C_0^{q_k}}\,.
    \]
    By Lemma \ref{l:max}, for every $k \ne 3$ and $c \in \{-2,-1,0,1,2\}$, 
    \[
        |\{M^{q_k}_{c\tau} \mathbf{1}_{E_k} > C_0 (m |E_k|)^{1/q_k}\}| \le  \frac{C}{C_0^{q_k}}\,.
    \]
    Choosing the constant $C_0$ sufficiently large, it follows that $|F| \le 1/12$, as required. $F$ is open because all the maximal functions are lower semi-continuous. It remains to verify the size estimate \eqref{e:size_est}. We fix $k \in \{1,2,3\}$ and write
    \[
        g=|E_k|^{-1/q_k}\mathbf{1}_{E_k}.
    \]
    When $k=i$ then Lemma \ref{l:iequalsk}, Jensen's inequality and the definition of $F$ yield 
    \[
        \size_{i,3,k}(g,P') \le \sup_{\mathbf{p} \in P'} \inf_{x \in I_{p_3}} M_{c\tau}g(x) \le C \sup_{\mathbf{p} \in P'} \inf_{x \in I_{p_3}} M_{c\tau}^{q_k} g(x) \le \begin{cases} C m^{1/q_k} & \text{if $k = 1,2$}\\ C & \text{if $k = 3$}\,.\end{cases}
    \]
    When $k=1,2$, the last inequality follows from $I_{p_3} \not\subset F$, while for $k = 3$ it holds since then $\|g\|_\infty \le 2$.

    When $k = 3 \ne i$ then by \eqref{e:square3} of Lemma \ref{l:jneqk} and since $|E_j|\ge 1/2$ it holds that
    \[
        \size_{i,3,k}(g, P') \le M^{2}g  \le 2 M^{2}\mathbf{1}_{E_j}  \le 2\,.
    \]
    When $k = 2 \ne i$ then by \eqref{e:square2}, the definition of $F$ and since $q_k = 2$, we have 
    \[
        \size_{i,3,k}(g, P') \le \sup_{\mathbf{p} \in P'} \inf_{x \in I_{p_3}} ( m^{1/2} M_{c\tau} g(x) + M^2 g(x)) \le Cm\,.
    \]
    There remains the case $k = 1 \ne i$. By \eqref{e:square1} we have for every $(i,3)$-tree $T$ in $P$
    \[
        \frac{1}{|I_T|} \sum_{\mathbf{p} \in T} |\langle g, \psi_{\mathbf{p},k}\rangle|^2 \le C(\sup_{\mathbf{p} \in T} \inf_{x \in I_{p_3}} m  (M_{c\tau} g(x))^2 + \frac{1}{|I_T|} \sum_{\substack{\mathbf{p} \in T\\ |I_{p_k}|\le 2^{-m}|I_T|}} |\langle g, \psi_{\mathbf{p},k} \rangle|^2)\,. 
    \]
    By definition of $F$, the first summand on the right hand side is bounded by $Cm^3$, and we may focus on the second term
    \begin{equation}\label{e:123}
        \frac{1}{|I_T|} \sum_{\substack{\mathbf{p} \in T\\ |I_{p_k}|\le 2^{-m}|I_T|}} |\langle g, \psi_{\mathbf{p},k} \rangle|^2\,.
    \end{equation}
    We consider first the contribution of tiles with $I_{p_k} \subset I_T$. If $\mathbf{p}'$ is a tri-tile with
    \[
        \frac{1}{|I_{p_k'}|}\int_{I_{p_k'}} |g| > C_1 m,
    \]
    then for all $x \in I_{p_3'}$ and some $c \in \{-2,-1,0,1,2\}$ it holds that
    \[
        M_{c \tau}g(x)\ge C C_1 m. 
    \]
    Choosing $C_1$ sufficiently large, this implies that $I_{p_3'} \subset F$. It follows that $I_{p_3} \not\subset F_{C_1m}$ for all $\mathbf{p} \in P'$, where $F_{C_1m}$ is the set defined in Proposition \ref{p:ref}.
    By modulating $g$ with $e^{2\pi i \xi'_{T}}$, where $\xi_{T}'$ is the frequency given by Lemma \ref{l:lac}, we may assume that all the wave packets in \eqref{e:123} have mean zero. Then we are in the situation of Proposition \ref{p:ref}, which gives
    \[
        \frac{1}{|I_T|} \sum_{\substack{\mathbf{p}\in T, \ I_{p_k} \subset I_T\\ |I_{p_k}|\le 2^{-m}|I_T|}} |\langle g,\psi_{\mathbf{p},k} \rangle |^2\le C m^2 \, .
    \]
    It remains to consider tiles with $I_{p_k} \not\subset  I_T$. Then $I_{p_3}\subset I_T$ but $I_{p_k} = I_{p_3}+c\tau_{s(\mathbf{p})}|I_T| \not \subset I_T$. The measure of the union or such intervals $I_{p_k}$ with $|I_{p_k}| = 2^{-d}|I_T|$ is at most $C2^{m-d}|I_T|$. By Lemma \ref{l:iequalsk}, their contribution to \eqref{e:123} is then bounded by 
    \[
        C \frac{1}{|I_T|}\sum_{\substack{\mathbf{p}\in T, \ I_{p_k} \not\subset I_T\\ |I_{p_k}|\le 2^{-m}|I_T|}} |I_{p_k}| \inf_{x \in I_{p_3}} (M_{c\tau} g(x))^2 \le C m^2 \sum_{d \ge m} 2^{m-d}  \le C m^2\,.
    \]
    This completes the proof.
\end{proof}

\section{Sharp variation bounds for   ergodic averages}
\label{s:ergodic}

In this section we prove Theorem \ref{t:long} and Theorem \ref{t:short}. As discussed in Section \ref{ss:ergodic}, this implies Theorem \ref{t:ergodic_long} and Theorem \ref{t:ergodic_short} by standard transference arguments. 

For technical reasons we will in some arguments in this section require control of more Schwartz semi-norms than what is given in $\mathcal{S}_0$. Denote by $\mathcal{S}_{0,+}$ the space of all mean zero functions satisfying 
\[
    \sup_{x} |x|^m |\psi^{(n)}(x)| \le 1, \qquad\qquad 0 \le m, n \le 200\,,
\]
and by $\mathcal{S}_{0,+}^{2^j}$ the set of $2^j$-shifts of functions in $\mathcal{S}_{0,+}$.

Our arguments rely on the following decomposition of $\mathbf{1}_{[0,1]}$. 
\begin{lemma} \label{l:dec_ind}
    There exists $C > 0$ and functions 
    \begin{equation}\label{e:ind_adap}
       \phi_{0,j} \in C\mathcal{S}_{0,+}, \qquad\qquad \phi_{1,j} \in C\mathcal{S}_{0,+}^{2^j}\,, \qquad\qquad j \ge 1,
    \end{equation}
    and a Schwartz function $\phi:\R\to \C$, 
    such that 
    \begin{equation}\label{e:ind_dec}
        \mathbf{1}_{[0,1]} = \phi + \sum_{j = 1}^\infty 2^{-j} D_{2^{-j}} \phi_{0, j} + \sum_{j = 1}^\infty 2^{-j} D_{2^{-j}}\phi_{1, j}\,.
    \end{equation} 
\end{lemma}

\begin{proof}
Let $\chi$ and $\rho$ be smooth functions, $\widehat{\chi}$ supported in $[-2,2]$ and $\widehat{\rho}$ supported in  $[-2,-1/2]\cup[1/2,2]$  such that for all $\xi\in \R$, 
\begin{equation}\label{e:partitionunity}
    \widehat{\chi}(\xi)+ \sum_{j=1}^\infty \widehat{\rho}(2^{-j}\xi)  =1\,.
\end{equation}   
Let for $j \ge 1$ 
\[
    \phi_{0,j}(x) = \int_{-\infty}^x \rho(t) \, dt
\]
be the primitive of ${\rho}$. Set further
\[\phi = \mathbf{1}_{[0,1]}\ast{\chi},\]
and for $j\ge 1$ set
\[
    \phi_{1,j}(x) =  - \phi_{0,j}(x-2^j)\,. 
\]
Then \eqref{e:ind_adap} holds since the primitive of $\rho$ is in $C\mathcal{S}_{0,+}$.
For \eqref{e:ind_dec}, we compute from \eqref{e:partitionunity}
\begin{align*}
   \mathbf{1}_{[0,1)} &= \mathbf{1}_{[0,1)}\ast{\chi} + \sum_{j=1}^{\infty} \mathbf{1}_{[0,\infty)}\ast D_{2^{-j}}{\rho}-\sum_{j=1}^{\infty} \mathbf{1}_{[1,\infty)}\ast D_{2^{-j}}{\rho}\\
    &=\phi + \sum_{j=1}^{\infty}2^{-j} D_{2^{-j}}\phi_{0,j} + \sum_{j=1}^{\infty} 2^{-j} D_{2^{-j}} \phi_{1,j}\,.
\end{align*}
This completes the proof. 
\end{proof}

\subsection{Long variation: Proof of Theorem \ref{t:long}}

We use the following result of \cite{DOP15}. 

\begin{theorem}[{\cite[Theorem 1.3]{DOP15}}]\label{t:smooth}
    Let $r>2$ and let $p,p_1,p_2$ be exponents satisfying \eqref{e:p_aspt}. 
    For every Schwartz function $\varphi: \R \to \C$ there exists a constant $C>0$ such that for every $f_1\in L^{p_1}(\R)$ and $f_2\in L^{p_2}(\R)$, 
    \[
        \|B_{2^s}(\varphi,f_1,f_2)(x)\|_{L_x^p(V^r_s(\Z))}\le C \|f_1\|_{p_1}\|f_2\|_{p_2}.
    \]
\end{theorem}

In fact, Theorem 1.3 in \cite{DOP15} bounds the full variation, not just the long variation, but we will not need this.
From the shifted Bilinear Hilbert transform estimate in Theorem \ref{thm:main2} we further obtain the following. 

\begin{lemma}\label{l:shifted_variation}
    Let $r>2$ and let $p,p_1,p_2$ be exponents satisfying \eqref{e:p_aspt}. 
    There exists a constant $C>0$ such that for every $j \ge 1$, every $\psi_j \in \mathcal{S}_0^j$ and every $f_1\in L^{p_1}(\R)$ and $f_2\in L^{p_2}(\R)$, 
    \[
        \|B_{2^s}(\psi_j,f_1,f_2)(x)\|_{L_x^p(V^r_s(\Z))}\le C j^{4} \|f_1\|_{p_1}\|f_2\|_{p_2}.
    \]
\end{lemma}

\begin{proof}  
    We denote $B_{2^s}(\psi_j)=B_{2^s}(\psi_j,f_1,f_2)$.
    Since $r > 2$, monotonicity of $V^p$ norms and the triangle inequality yield
    \[
        \|B_{2^s}(\psi_j)(x)\|_{L_x^p(V^r_s(\Z))}^p \le \|B_{2^s}(\psi_j)(x)\|_{L_x^p(V^2_s(\Z))}^p      
        \le C\|B_{2^s}(\psi_j)(x)\|_{L_x^p(\ell^2_s(\Z))}^p\,.
    \]
    By Khintchine's inequality, this is with random Rademacher signs $\varepsilon \in \{-1,1\}^\mathbb{Z}$
    \[
        \le C\, \mathbb{E}_\varepsilon \Big\| \sum_{s \in \mathbb{Z}}  B_{2^s}(\varepsilon_s\psi_j)\Big\|_{p}^p\,.
    \]
    Theorem \ref{thm:main2} and convexity now complete the proof of the lemma.
\end{proof}

In the case $p \ge 1$, Theorem \ref{t:long} follows by applying the decomposition from Lemma \ref{l:dec_ind} to $\mathbf{1}_{[0,1]}$, and then estimating using the triangle inequality,  Theorem \ref{t:smooth},  and Lemma \ref{l:shifted_variation}. When $p < 1$, one replaces the triangle inequality by the $p$-triangle inequality 
\begin{equation} \label{pnormineq}
    \Big\|\sum_j g_j\Big\|_p^p \le \sum_j \|g_j\|_p^p\,.
\end{equation}

\subsection{Short variation: Proof of Theorem \ref{t:short}} 
We will deduce the short variation estimates from a continuous square function estimate. 
For Schwartz functions $f_1,f_2$, and $\psi$, with $\widehat{\psi}(0)=0$, we denote 
\[
    S(\psi,f_1,f_2)(x) = \Big(\int_{0}^{\infty}  |B_{t}(\psi, f_1,f_2)(x)|^2 \frac{dt}{t}\Big)^{1/2}\,. 
\] 
Theorem \ref{thm:main2} implies the following estimate for $S$.

\begin{lemma}\label{l:square1}
    Let $p,p_1, p_2$ be exponents satisfying \eqref{e:p_aspt} with $p \ge 2$. There exists a constant $C > 0$ such that for every $j \ge 1$, every  $\psi \in \mathcal{S}_0^{2^j}$, and all Schwartz functions $f_1, f_2$
    \[
        \| S(\psi,f_1,f_2) \|_p \le C j^{4} \|f_1\|_{p_1} \|f_2\|_{p_2}\,.
    \]
\end{lemma}

\begin{proof}
    We fix $f_1, f_2$ and write $B_t(\psi) = B_t(\psi, f_1, f_2)$.
    We expand the $L^p$ norm, split the integral up and make a change of variables in $t$ 
    \[
        \|S(\psi, f_1, f_2)\|_p = \Big(\int_\R \Big( \int_1^2 \sum_{s \in \mathbb{Z}} |B_{2^st}(\psi)(x)|^2 \, \frac{dt}{t} \Big)^{p/2} \, dx\Big)^{1/p}\,.
    \]
    Minkowski's integral inequality, using that $p/2 \ge 1$, estimates this by 
    \[
        \Big(\int_1^2 \Big(\int_\R \Big(\sum_{s \in \mathbb{Z}} |B_{2^st}(\psi)(x)|^2\Big)^{p/2} \, \mathrm{d}x \Big)^{2/p} \frac{dt}{t}\Big)^{1/2} = \Big(\int_1^2 \|B_{2^st}(\psi)(x)\|_{L^p_x \ell^2_s(\mathbb{Z})}^2  \frac{dt}{t}\Big)^{1/2}\,.
    \]
    For any $t \in [1,2]$, Theorem \ref{thm:main2} with  Khintchine's inequality as in the proof of Lemma \ref{l:shifted_variation} and a scaling argument imply
    \[
        \|B_{2^st}(\psi)(x)\|_{L^p_x \ell^2_s(\mathbb{Z})} \le Cj^4 \|f_1\|_{p_1} \|f_2\|_{p_2}\,.
    \]
    This completes the proof.
\end{proof}

For $p < 2$ the application of Minkowski's inequality in the above argument fails. Heuristically, the function $t \mapsto B_t(\psi, f_1, f_2)$ is smooth at scale $2^{-j}$, which suggests a loss of a factor $2^{j(2/p - 1)}$.
The following lemma makes this heuristic precise. 

\begin{lemma}\label{l:square2}
   Let $p,p_1, p_2$ be exponents satisfying \eqref{e:p_aspt} with $p < 2$. There exists a constant $C > 0$ such that for every $j \ge 1$, every  $\psi \in \mathcal{S}_{0,+}^{2^j}$, and all Schwartz functions $f_1, f_2$
    \[
        \| S(\psi,f_1,f_2) \|_p \le C j^{4} 2^{j(1/p - 1/2)} \|f_1\|_{p_1} \|f_2\|_{p_2}\,.
    \]
\end{lemma}

\begin{proof}
    Let $\rho$ be a smooth function supported in $[0,2^{-j}]$ such that for all $t \in \R$
    \[
        \sum_{s \in \mathbb{Z}} \rho(t - 2^{-j-1}s)^2 = 1\,.
    \]
    Using the integration variable $u = \log_2 t$ in the square function yields
    \[
        \|S(\psi,f_1,f_2)\|_p^p =  (\log 2)^p \int_\R \Big( \sum_{s \in \mathbb{Z}} \int_\R \rho(u - s 2^{-j-1})^2 |B_{2^u}(\psi, f_1, f_2)(x)|^2 du \Big)^{p/2} \, dx\,.
    \]
    Splitting the summation in $s$ into congruence classes modulo $2^{j + 1}$ and using the sub-additivity of $x \mapsto x^{p/2}$ for $p < 2$ yields
    \begin{align*}
        &\le C\sum_{\ell = 0}^{2^{j+1}-1} \int_\R \Big(  \sum_{s \equiv \ell \!\!\!\!\pmod{2^{j+1}}} \int_\R \rho(u - s 2^{-j-1})^2 |B_{2^u}(\psi, f_1, f_2)(x)|^2 \, du \Big)^{p/2} \, dx\,.
    \end{align*}
    By a scaling argument, it suffices to estimate the $\ell = 0$ summand by
    \begin{equation}\label{e:int_2_a}
        \int_\R \Big(  \sum_{s\in \mathbb{Z}} \int_\R \rho(u - s)^2 |B_{2^u}(\psi, f_1, f_2)(x)|^2 \, du \Big)^{p/2} \, dx \le C2^{-jp/2} j^{4p} \|f_1\|_{p_1}^p \|f_2\|_{p_2}^p\,.
    \end{equation}
    Since $B_{2^u}(\psi, f_1, f_2)(x)$ is morally constant in $u$ at scale $2^{-j}$, and the function $\rho$ is supported on an interval of length $2^{-j}$, one expects the integrand to be essentially equal to $2^{-j}|B_{2^s}(\psi, f_1, f_2)|^2$. If this was an exact equality, then \eqref{e:int_2_a} would follow from Theorem \ref{thm:main2}. We will apply Plancherel's theorem and verify that the contribution of the non-constant Fourier modes is negligible, making this almost constant property precise.  
    
    Applying Plancherel's theorem to the $u$ integral in \eqref{e:int_2_a}, using that the integrand is supported on an interval of length $2^{-j}$, yields for every $s \in \mathbb{Z}$   
    \begin{equation}\label{e:post_plan}
       \sum_{\ell \in \mathbb{Z}}  2^j \Big|\int_\R e^{-2\pi i 2^j \ell u} \rho(u - s) B_{2^u}(\psi, f_1, f_2)(x) \, \mathrm{d}t \Big|^2  = 2^{-j}\sum_{\ell \in \mathbb{Z}}  |B_{2^s}(\Psi_\ell, f_1, f_2)(x)|^2 \,,
    \end{equation}
    where by expanding and changing variables $u \mapsto 2^{-j}u$ one finds that for $\ell \in \mathbb{Z}$
    \begin{align*}
        \Psi_\ell(y) &= \int_\R e^{-2 \pi i \ell u} \rho(2^{-j}u) 2^{-2^{-j}t} \psi(2^{-2^{-j}t} y) \, dt\,.
    \end{align*}
    The key claim is that 
    \begin{equation}\label{e:claim Psi}
        (1 + |\ell|)^{10} \Psi_\ell = \overline{\Psi}_\ell \in \mathcal{S}_0^{2^j}\,.
    \end{equation}
    This follows from a routine integration by parts argument, which we outline below. Assuming the claim for now, we use \eqref{e:post_plan} and the definition of $\overline{\Psi}_\ell$ to estimate the left hand side of \eqref{e:int_2_a} by
    \[
        2^{-jp/2}\sum_{\ell \in \mathbb{Z}} (1 + |\ell|)^{-10 p / 2} \int \Big( \sum_{s \in \mathbb{Z}} |B_{2^{s}}(\overline{\Psi}_{\ell}, f_1, f_2)(x)|^2 \Big)^{p/2} \, dx\,.
    \]
    Then \eqref{e:int_2_a} follows from Theorem \ref{thm:main2} and Khintchine's inequality,  as in the proof of Lemma \ref{l:shifted_variation}.

    It remains to prove the claim \eqref{e:claim Psi}. The function
    \[
        \rho_0(t) = \rho(2^{-j}t)2^{-2^{-j}t}
    \]
    is supported in $[0,1]$ and each of its derivatives bounded uniformly in $j$. We have
    \begin{equation}\label{e:oscillatory_integral}
        \Psi_\ell(y) = \int_0^1 e^{-2\pi i \ell t} \rho_0(t) \psi(2^{-2^{-j}t} y) \, dt\,.
    \end{equation}
    Recall that we assume $\psi \in \mathcal{S}_{0,+}^{2^j}$. Hence, the function $\psi$ has mean $0$ and so does $\Psi_\ell$ for each $\ell$. 
    It further holds for all $m \le 100$ that
    \[
        |\partial^m \psi(x)| \le (1 + |x - 2^j|)^{-200}\,.
    \]
    We will integrate  by parts in $t$ in \eqref{e:oscillatory_integral} to obtain the claimed decay in $\ell$. For this purpose we further record that for $k \le 10$ and $t \in [0,1]$
    \begin{align*}
        |\partial^k_t \psi(2^{-2^{-j}t}y)| &\le C(1 + |2^{-j}2^{-2^{-j}t}y|^k) (1 + ||2^{-2^{-j}t}y| - 2^j|)^{-200}\\
        &\le C (1 + |y - 2^j|)^{-100}\,.
    \end{align*}
    Integration by parts and the support of $\rho_0$ in $[0,1]$ then yields  
    \begin{equation}\label{e:Psi_decay}
        |\Psi_{\ell}(y)| \le C (1 + 2 \pi |\ell|)^{-10} (1 + |y - 2^j|)^{-100}\,,
    \end{equation}
    as required. To control derivatives of $\Psi_\ell$, recall that 
    \[
        \partial_y^k \Psi_{\ell}(y) = \int_0^1 e^{-2\pi i \ell t} \rho_0(t)2^{-k2^{-j}t} (\partial^k \psi)(2^{-2^{-j}t}y) \, dt\,.
    \]
    All the estimates used in the above proof continue to hold uniformly in $k \le 100$ for the function
    \[
        \rho_k(t) = \rho_0(t) e^{-k2^{-j}t}
    \]
    and for $\partial^k \psi$. We conclude that \eqref{e:Psi_decay} also holds for the first $100$ derivatives of $\Psi_\ell$, which completes the proof.    
\end{proof}

We will deduce Theorem \ref{t:short} from Lemmas \ref{l:square1} and \ref{l:square2}, using the following consequence of the fundamental theorem of calculus.  

\begin{lemma}\label{l:variation_FTC}
    There exists a constant $C> 0$ such that for every Lipschitz function $a: [2^s, 2^{s+1}] \to \C$ and every $r \in [1,2]$
    \begin{equation}
        \label{e:ftc2}
        \|a\|_{V^r([2^s, 2^{s+1}])}^2 \le C \Big(\int_{2^s}^{2^{s+1}} |a(t)|^2 \frac{dt}{t}\Big)^{1/r'} \Big( \int_{2^s}^{2^{s+1}} |ta'(t)|^2 \frac{dt}{t}\Big)^{1/r} \,.
    \end{equation}
\end{lemma}

\begin{proof}
    By H\"older's inequality, it suffices to show this for $r = 1$ and $r = 2$. For $r = 1$ it follows directly from the fundamental theorem of calculus and H\"older's inequality. We now prove the $r = 2$ case.

    By splitting into the real and imaginary parts, and then into positive and negative parts, we may assume that $a \ge 0$. Pick a sequence $2^s \le t_0 < \dotsb < t_J \le 2^{s+1}$ attaining the supremum in the $2$-variation up to a factor of $2$. Then 
    \[
        \|a\|_{V^2([2^s, 2^{s+1}])}^2 \le 2 \sum_{j=1}^J |a_{t_j} - a_{t_{j-1}}|^2 \le 2 \sum_{j=1}^J |a_{t_j}^2 - a_{t_{j-1}}^2|\,.
    \]
    By the fundamental theorem of calculus, this is at most
    \[
        4 \sum_{j=1}^J \int_{t_{j-1}}^{t_j}  \lvert a'(t)\rvert \lvert a(t)\rvert \, dt \le 4\int_{2^s}^{2^{s+1}}\lvert a'(t) \rvert \lvert a(t)\rvert \, dt\,.
    \]
    The $r = 2$ case of the lemma follows by the Cauchy-Schwarz inequality.
\end{proof}

Now we complete the proof of Theorem \ref{t:short}.
We fix the Schwartz functions $f_1,f_2$, normalized so that 
\[
    \|f_1\|_{p_1}=\|f_2\|_{p_2}=1
\]
and for a function $\varphi$ we write $B_{t}(\varphi)(x)=B_{t}(\varphi, f_1,f_2)(x)$. In the proof we  will  apply Lemma \ref{l:variation_FTC} to the bilinear averages of the form $a(t)=B_{t}(\varphi)(x)$, in which case 
\[
    t\partial_t B_{t}(\varphi)(x) = -B_{t}(\psi)(x).
\]
with  $\psi(u) = (u \varphi(u))'$. 

Let $\phi$ be the function given by Lemma \ref{l:dec_ind}.
It suffices to prove 
\begin{equation}
    \label{e:short_phi}
    \Big\| \Big(\sum_{s \in \mathbb{Z}} \|B_t(\phi)\|_{V^r_t([2^s, 2^{s+1}])}^2\Big)^{1/2}\Big\|_{p}\le C 
\end{equation}
and
\begin{equation}
    \label{e:short_1-phi}
    \Big\| \Big(\sum_{s \in \mathbb{Z}} \|B_t(\mathbf{1}_{[0,1]}-\phi)\|_{V^r_t([2^s, 2^{s+1}])}^2\Big)^{1/2}\Big\|_{p}\le C\,.
\end{equation}
Using the case $r = 1$ of Lemma \ref{l:variation_FTC} with $\varphi = \phi$ and summing in $s$ yields
\[ 
    \sum_{s \in \mathbb{Z}} \|B_t(\phi)\|_{V^r_t([2^s, 2^{s+1}])}^2 \le \sum_{s \in \mathbb{Z}} \|B_t(\phi)\|_{V^1_t([2^s, 2^{s+1}])}^2 \le C S(\psi)^2
\]
with $\psi(u) = (u \phi(u))' \in C\mathcal{S}_0$, where $S(\psi)=S(\psi,f_1,f_2)$. Then Lemma \ref{l:square1} yields \eqref{e:short_phi}. 
For \eqref{e:short_1-phi}, we first assume that $p \ge 1$. Lemma \ref{l:dec_ind} and a scaling argument yield
\[
     \Big\| \Big(\sum_{s \in \mathbb{Z}} \|B_t(\mathbf{1}_{[0,1]}-\phi)\|_{V^r_t([2^s, 2^{s+1}])}^2\Big)^{1/2}\Big\|_{p} 
    \le \sum_{\iota=0}^1 \sum_{j=1}^\infty 2^{-j} \Big\| \Big(\sum_{s \in \mathbb{Z}} \|B_t(\phi_{\iota,j})\|_{V^r_t([2^s, 2^{s+1}])}^2\Big)^{1/2}\Big\|_{p}\,.
\]
Applying Lemma \ref{l:variation_FTC} and denoting $\psi_{\iota, j}(t) = (t \phi_{\iota, j}(t))'$ yields  
\begin{align}
    \sum_{s \in \mathbb{Z}} \|B_t(\phi_{\iota,j})\|_{V^r_t([2^s, 2^{s+1}])}^2
    &\le C \sum_{s \in \mathbb{Z}} \Big(\int_{2^s}^{2^{s+1}}  |B_{t}(\phi_{\iota,j})|^2 \frac{dt}{t}\Big)^{1/r'} \Big(\int_{2^s}^{2^{s+1}}  |B_{t}(\psi_{\iota,j})|^2 \frac{dt}{t}\Big)^{1/r}
    \nonumber\\
    &\le C S(\phi_{\iota,j})^{2/r'} S(\psi_{\iota,j})^{2/r}\,, \label{e:FTC_B}
\end{align}
where the second inequality follows by  H\"older's inequality in $s$. 
It follows that
\[
     \Big\| \Big(\sum_{s \in \mathbb{Z}} \|B_t(\mathbf{1}_{[0,1]}-\phi)\|_{V^r_t([2^s, 2^{s+1}])}^2\Big)^{1/2}\Big\|_{p} \le C\sum_{\iota=0}^1\sum_{j  =1}^\infty 2^{-j} \| S(\phi_{\iota,j})^{1/r'} S(\psi_{\iota,j})^{1/r} \|_{p}\,.
\]
Applying Hölder's inequality in each summand, this is
\[
    \le 
    C\sum_{\iota=0}^1\sum_{j  =1}^\infty 2^{-j} \| S(\phi_{\iota,j}) \|_p^{1/r'} \| S(\psi_{\iota,j})\|_{p}^{1/r}
.
\]
Recall from Lemma \ref{l:dec_ind} that $\phi_{0,j}, \psi_{0,j} \in C\mathcal{S}_0$ and $\phi_{1,j}, 2^{-j} \psi_{1,j} \in C\mathcal{S}_0^{2^j}$. If $p \ge 2$ we use Lemma \ref{l:square1}, to obtain 
\begin{equation*}
    \eqref{e:short_1-phi} \le C\sum_{j=1}^\infty (2^{-j} + 2^{-j/r'}j^4)\,. 
\end{equation*}
Since $r > 1$, this completes the proof when $p \ge 2$. When $p \in [1,2)$, we use instead Lemma \ref{l:square2}, which yields
\begin{equation*}
     \eqref{e:short_1-phi} \le C  \sum_{j  =1}^\infty (2^{-j} +    2^{-j/r'}j^4) 2^{j(1/p - 1/2)}\,. 
\end{equation*}
This is finite as we assume that $1/r < \min\{3/2 - 1/p, 1\}$. Finally, when $p < 1$, we replace the triangle inequality in the above computation by the $p$-triangle inequality \eqref{pnormineq}. This replaces the final sum by the $\ell^p$ sum of the same summands, which still converges. This completes the proof.

\section{The bilinear Hilbert transform with Dini continuous kernel}

Here we prove Theorem \ref{t:Dini}. We fix a modulus of continuity $\eta$ satisfying \eqref{e:eta_aspt} and a kernel $K$ satisfying \eqref{e:K_eta}.
Let $\rho$ be an smooth, even function supported in
$[-2,-1/2]\cup [1/2,2]$ such that  
\begin{align}  
    \sum_{s \in \mathbb{Z}} \rho_s(x) = 1, \quad x \ne 0, \qquad\qquad \rho_s(x) = \rho(2^{-s} x)\,. \label{e:lr_psi1}
\end{align}
Let further $\zeta$ be a smooth, even, mean zero function supported in $[-1,1]$, normalized such that for all $\xi \ne 0$
\[
    \int_0^\infty |\widehat \zeta(t \xi)|^2 \, \frac{dt}{t} = 1\,,
\]
and set $\zeta_t = D_t \zeta$.
Using Fourier inversion, one can then decompose the kernel $K$ as 
\begin{equation}\label{e:cont_dec}
    K = \sum_{s \in \mathbb{Z}} \int_0^\infty K_s^t * \zeta_{t} \, \frac{dt}{t}, \qquad\qquad K_s^t = (K \rho_{s}) * \zeta_{t}\,.
\end{equation}
The following identity shows how this construction behaves under scaling 
\begin{equation}
    \label{scalingid}
    D_{2^{-s}}K_s^t = ((D_{2^{-s}}K)\rho)*D_{t2^{-s}}\zeta\,.
\end{equation}
The next lemma verifies that the `smooth part', where $t \ge 2^s$, satisfies the assumptions of Theorem \ref{thm:main2} with shift $\tau = 0$.

\begin{lemma}\label{l:lr_K0}
    There exists a constant $C > 0$ such that for each $s\in \Z$    
    \[
        \int_{2^s}^\infty K_s^t * \zeta_{t}\, \frac{dt}{t} = \eta(1) D_{2^{s}} \phi_{s,0},\qquad \qquad \phi_{s,0} \in C \mathcal{S}_0 .
    \] 
\end{lemma}

\begin{proof}
Since the class of kernels satisfying \eqref{e:K_eta} is invariant under dilations $D_\lambda$, by \eqref{scalingid} it suffices to show the claim for $s = 0$. 
 
Let $K_0 = K\rho$, this is an odd function. It holds that 
\[
     \int_{1}^\infty K_0^t * \zeta_{t} \, \frac{dt}{t} = K_0 * \int_{1}^\infty \zeta_{t} * \zeta_{t} \, \frac{dt}{t} = K_0 * \varphi\,.
\]
Here
\[
    \widehat \varphi(\xi) = \int_{1}^\infty |\widehat\zeta(t\xi)|^2 \, \frac{dt}{t}
\]
is a Schwartz function, so $\varphi$ is a Schwartz function as well. As $K_0$ is bounded by $C\eta(1)$, has mean $0$, and is compactly supported, this completes the proof.
\end{proof}

We further split the `non-smooth' parts into pieces to which we can apply Theorem~\ref{thm:main2}.

\begin{lemma}\label{l:lr_Km}
    There exists a constant $C>0$ such that for each $m \ge 1$ and each $s\in \Z$ there exists a decomposition 
    \begin{equation}\label{e:lr_dec}
        \int_{2^{s - m}}^{2^{s-m+1}} K_s^{t} * \zeta_{t} \, \frac{dt}{t} = \sum_{\ell = - 10 \cdot 2^m}^{10 \cdot 2^m} c_{s,m,\ell} D_{2^{s-m}} \phi_{s,m,\ell}, \qquad\qquad  \phi_{s,m,\ell}(x) \in C\mathcal{S}_0^{\ell}, 
    \end{equation}
    where for all $s, m ,\ell$, 
    \begin{equation}\label{e:coeff_bound}
        |c_{s,m,\ell}| \le  
    2^{-m}\eta(2^{-m})\,.
    \end{equation}
\end{lemma}

\begin{proof}
By scaling, recall \eqref{scalingid}, we may again assume  that $s = 0$. 

First we verify that $K_0 = K\rho$  has modulus of continuity $C\eta$. Recall that $K_0$ is a bounded odd function supported in $1/2 < |x| < 2$. Using this and smoothness of $\rho$,  we have for all $x,x'$ in the support of $\rho$
\begin{align}    
    |K_0(x) - K_0(x')| &\le |K(x)| |\rho(x)- \rho(x')| + |\rho(x')| |K(x) - K(x')|\nonumber\\
    &\le C |x - x'| \eta(1) + \eta(|x-x'|)\nonumber\\
    &\le C \eta(|x- x'|).\label{e:K0_dini}
\end{align}
Here we also used that by \eqref{e:eta_aspt} for each $0< t <1$ it holds $t \le 2\eta(1)^{-1} \eta(t)$. 
Clearly this continues to hold for $x, x'$ both outside the support of $\rho$. In the remaining case we have without loss of generality that $x$ is in the support of $\rho$ and $x'$ is not. Then only the first term in the above computation is not zero and \eqref{e:K0_dini} still follows. 

Using \eqref{e:K0_dini} and that $|y| \le t$ on the support of $\zeta_t$, it follows that for all $t \in [2^{-m}, 2^{1-m}]$
\begin{equation}\label{e:K0size}     
    |K_0^t(x)| \le \int |K_0(x - y) - K_0(x)| |\zeta_{t}(y)| \, dy \le C \eta(2^{-m})\,.
\end{equation}
Note also that $K_0^t$ is supported in $[-4,4]$ for all $t \in [2^{-m}, 2^{1-m}]$. Moreover, $\zeta_t \in C D_{2^{-m}} \mathcal{S}_0$ for every $t \in [2^{-m}, 2^{1-m}]$, hence its shifts satisfy
\begin{equation}\label{e:shift_good}
    |2^my - \ell| \le 1/2 \qquad\qquad \implies \qquad \qquad \zeta_{t}(\cdot - y) \in C D_{2^{-m}} \mathcal{S}_0^\ell\,.
\end{equation}
We define the coefficients
\[
    c_{0,m,\ell} = \int_{|2^my - \ell| \le 1/2} \int_{2^{-m}}^{2^{1-m}} |K_0^t(y)| \, \frac{dt}{t} \, dy
\]
and the functions
\[
    D_{2^{-m}} \phi_{0, m, \ell} = \frac{1}{c_{0,m,\ell}} \int_{|2^my - \ell| \le 1/2} \int_{2^{-m}}^{2^{1-m}} K_0^t(y) \zeta_{t}(x-y) \, \frac{dt}{t} \, dy\,.
\]
By \eqref{e:shift_good} and convexity, we obtain \eqref{e:lr_dec}. From the support of $K_0^t$ and \eqref{e:K0size} it follows that for $|\ell| \le 10 \cdot 2^m$
\[
   |c_{0,m,\ell}| \le C 2^{-m} \eta(2^{-m})
\]
and that $|c_{0,m,\ell}| = 0$ for larger values of $|\ell|$. This completes the proof.
\end{proof}

\begin{proof}[Proof of Theorem \ref{t:Dini}] 

We summarize our decompositions from \eqref{e:cont_dec} and Lemmas \ref{l:lr_K0} and \ref{l:lr_Km} as 
\begin{align}
    {B}(K,f_1,f_2) &= \sum_{s \in \mathbb{Z}} \eta(1)B_{2^s}(\phi_{s,0}, f, g)+ \sum_{s \in \mathbb{Z}} \sum_{m \ge 1} \sum_{\ell \in \mathbb{Z}} c_{s,m,\ell} B_{2^{s-m }}(\phi_{s,m,\ell}, f,g)\nonumber\\
    &= \sum_{s \in \mathbb{Z}} \eta(1)B_{2^s}(\phi_{s,0}, f, g)+ \sum_{\ell \in \mathbb{Z}} \sum_{a \in \mathbb{Z}} c_{a,\ell} B_{2^a}( \phi_{a, \ell},f,g)\label{e:dini_final_dec}
\end{align}
where
\[
    c_{a,\ell} = \sum_{s - m = a} |c_{s,m,\ell}|, \qquad\qquad \phi_{a, \ell} = \frac{1}{c_{a,\ell}} \sum_{s - m  = a} c_{s,m,\ell} \phi_{s,m,\ell}\,.
\]
By convexity of $C\mathcal{S}_0^{\ell}$, for all $a, \ell \in \mathbb{Z}$ it holds that $\phi_{a,\ell}\in C\mathcal{S}_0^{ \ell}$. Moreover, \eqref{e:coeff_bound} yields
\begin{align*}
    |c_{a,\ell}| = \sum_{m = 1}^\infty |c_{a+ m ,m, \ell}| &\le  C\sum_{2^m \ge |\ell|/10} 2^{-m} \eta(2^{-m}) \le C (1 + |\ell|)^{-1} \eta((1 + |\ell|)^{-1})\,.
\end{align*}
We assume by scaling that $\|f_1\|_{p_1} = \|f_2\|_{p_2}  =1$. Taking $L^p$ norms in \eqref{e:dini_final_dec}, applying the triangle inequality and then Theorem \ref{thm:main2} to each term yields when $p \ge 1$
\begin{align*}
    \|{B}(K,f_1,f_2) \|_p&\le C \Big(\eta(1) + \sum_{\ell\in \Z} (1 + |\ell|)^{-1} \eta( (1 + |\ell|)^{-1}) (1 +\lvert\log \ell \rvert)^{4}\Big)\\
    &\le C \int_0^1 \eta(t) \frac{\lvert\log t\rvert^{4}}{t} \, dt\,.
\end{align*}
For $p < 1$ the same argument with the $p$-triangle inequality \eqref{pnormineq} yields
\begin{align*}
    \|{B}(K,f_1,f_2) \|^p_p &\le C  \Big(\eta(1)^p + \sum_{\ell \in  \mathbb{Z}} \Big( (1 + |\ell|)^{-1} \eta( (1 + |\ell|)^{-1}) (1 +\lvert\log \ell \rvert)^{4}\Big)^p \Big)\\
    &\le C \int_0^1 \eta(t)^p \frac{\lvert\log t\rvert^{4p}}{t^{2-p}} \, dt\,.
\end{align*}
This completes the proof of Theorem \ref{t:Dini}.
\end{proof}

\section{The Hörmander multiplier theorem for the bilinear Hilbert transform}
\begin{proof}[Proof of Theorem \ref{t:Hörmander}]

Let $\rho$ and $\rho_s$ be as in \eqref{e:lr_psi1}. We fix the symbol $m$ and decompose it as  
\[
    m (\xi)  
    = \sum_{s \in \mathbb{Z}}  m_s (2^{s}\xi), \qquad\qquad m_s(\xi) = m(2^{-s}\xi)\rho_{-3}(\xi). 
\]
The function $m_s$ is supported in $[-1/4,1/4]$, so we may expand it as a Fourier series
\[
    m_s(\xi) = \sum_{\ell \in \mathbb{Z}} \widehat{m}_s(\ell)  e^{2\pi i\ell \xi} \,, \qquad \qquad \xi \in [- 1/2, 1/2]\,.
\]
By the support of $\rho_{-3}$ it also holds that $m_s = m_s \cdot \widehat \psi$, where $\widehat \psi = \rho_{-4} + \rho_{-3}+ \rho_{-2}$. Thus 
\begin{equation}\label{e:ms}
    m_s(\xi) = \sum_{\ell \in \mathbb{Z}}  \widehat m_s(\ell) e^{2\pi i\ell\xi} \widehat \psi(\xi)\,, \qquad\qquad  \xi \in \R\,.
\end{equation}
Define functions $\phi_{\ell}$, $\ell \in \mathbb{Z}$ via their Fourier transform as $\widehat{\phi}_{\ell}(\xi) = e^{2\pi i\ell \xi}\widehat \psi(\xi)$.
Then by the Fourier inversion formula and \eqref{e:ms}, the multiplier operator $B(\widecheck{m}, f_1,f_2)(x)$ equals
\begin{equation}\label{e:decm}
   \iint_{\R^2} \widehat{f_1}(\xi) \widehat{f_2}(\eta) m(\xi - \eta) e^{2\pi ix(\xi + \eta)} \, d\xi \, d\eta =  \sum_{s \in \mathbb{Z}} \sum_{\ell \in \mathbb{Z}} \widehat{m}_s(\ell)  B_{2^{s}}(\phi_{\ell}, f_1, f_2)\,.
\end{equation}
We normalize $\|f_1\|_{p_1} =\|f_2\|_{p_2} = 1$. Let first $p \ge 1$ and normalize $m$ so that 
\begin{equation}
    \label{mkcond}    
    \|m\|_{\mathscr{V}(Y_w)} = \sup_{s \in \Z} \sup_{a \ge 0} \frac{1 + a^5}{w_a} \int_{A(a)} |\widehat{m}_s(x)| \, dx = 1\,.
\end{equation}
We would like to pass this bound from the continuous Fourier transform of $m_s$ on to the Fourier coefficients. 
Using that $m_s = m_s \widehat{\psi}$, we have
\begin{equation*}
    \sum_{\ell \in A(a)} |\widehat m_s(\ell)| = \sum_{\ell\in A(a)} |\widehat m_s*\psi(\ell)| \le  \sum_{\ell \in A(a)}\int_{\R} \lvert\widehat{m}_s(x)\rvert\lvert\psi(\ell-x)\rvert \, dx\,.
\end{equation*}
Using that $\psi$ is a Schwartz function, in particular that $|\psi(x)| \le C(1 + |x|)^{-10}$, and \eqref{mkcond}, we obtain by splitting the integral in $x$ into integrals over the dyadic annuli $A(b)$
\begin{equation}\label{e:sum_ell}
    \sum_{\ell \in A(a)} |\widehat m_s(\ell)| \le C\sum_{b = 0}^\infty \frac{1}{(1 + |b - a|)^{10}} \frac{w_b}{1 + b^5} = \frac{W_a}{1 + a^5}\,,
\end{equation}
where we take this as the definition of $W_a$. The summability of $w_b$ gives 
\[
    \sum_{a =0}^\infty W_a = C\sum_{b = 0}^\infty w_b \sum_{a = 0}^\infty \frac{1}{(1 + |b - a|)^{10}} \frac{1 + a^5}{1 + b^5} \le C \sum_{b = 0}^\infty w_b \le C\,.
\]
To efficiently apply Theorem \ref{thm:main2}, we need the coefficients $\widehat{m}_s(\ell)$ to all be roughly of the same size. We can arrange this using dyadic pigeonholing. Define for $j \ge 0$ and $s \in \mathbb{Z}$ the set 
\[
    L(a, j, s) = \{ \ell \in A(a) : 2^{-j-1}(1 + a^{5})^{-1} W_a < |\widehat m_s(\ell)| \le 2^{-j} (1 + a^{5})^{-1} W_a\}\,. 
\]
By \eqref{e:sum_ell} it holds that $|L(a, j, s)| \le \min\{2^{j+1}, 2^{a}\}$.
We reorder the sum in \eqref{e:decm} as
\begin{align}
    &\quad\sum_{a = 0}^\infty \sum_{j =0}^\infty \sum_{s \in \mathbb{Z}} \sum_{\ell \in L(a, j, s)} \widehat{m}_s(\ell)  B_{2^{s}}(\phi_{\ell}, f_1, f_2)\nonumber\\
    &= \sum_{a = 0}^\infty \sum_{j =0}^\infty \sum_{t=1}^{\min\{2^{j+1},2^{a}\}}
    2^{-j} (1 + a^{5})^{-1} W_a  \sum_{s \in \mathbb{Z}} B_{2^{s}}(\widetilde{\phi^{a,j, t}_s}, f_1, f_2)\,,\label{e:final_reordering}
\end{align}
where we have enumerated the functions corresponding to $\ell \in L(a, j, s)$ by
\begin{equation}
    \label{def:phi_kas}
       \widetilde{\phi^{a, j, t}_s} = W_a^{-1} (1 + a^{5}) 2^{j} \widehat m_s(\ell) \phi_{\ell} \in C\mathcal{S}^{\ell}_0\,, \qquad\qquad \ell \in L(a, j, s)\,.
\end{equation}
Applying the triangle inequality in \eqref{e:final_reordering} and then Theorem \ref{thm:main2}, using that $|\ell| \sim 2^a$ for $\ell \in L(a, j, s)$, yields  
\[
  \|B(\widecheck{m},f_1,f_2)\|_p  \le C \sum_{a = 0}^\infty W_a \cdot (1 + a)^{-1} \sum_{j =0}^\infty \min\{1, 2^{a - j}\} \le C\,.
\]
This completes the proof in the case $p \ge 1$.

When $p < 1$, we normalize $m$ so that 
\[
    \sup_{s \in \mathbb{Z}} \sup_{a \ge 0} \frac{1 + a^{4p+1}}{w_a^p} \int_{A(a)} |\widehat{m}_s(x)|^p  \, dx = 1\,.
\]
There exists again another sequence $W$ depending only on $w$ such that
\begin{equation}\label{e:sum_ell_p}
    \sum_{\ell\in A(a)} |\widehat m_s(\ell)|^p \le C (1 +a^{4p+1})^{-1} W_a^p\,, \qquad\qquad \sum_{a = 0}^\infty W_a^p \le C\,.
\end{equation}
We now set
\[
    L(a, j, s) = \{ \ell \in A(a) : 2^{-j-1}( 1 + a^{4p + 1})^{-1/p} W_a < |\widehat m_s(\ell)| \le 2^{-j}( 1 + a^{4p + 1})^{-1/p}  W_a\}\,,
\]
and reorder the sum in \eqref{e:decm}  as 
\begin{equation}\label{e:m_reord_p}
     \sum_{a = 0}^\infty \sum_{j =0}^\infty \sum_{s \in \mathbb{Z}} \sum_{\ell \in L(a, j, s)} \widehat{m}_s(\ell)  B_{2^{s}}(\phi_{\ell}, f_1, f_2)\,.
\end{equation}
The same argument as for $p \ge 1$ with \eqref{pnormineq} in place of the triangle inequality yields
\[
   \|B(\widecheck{m},f_1,f_2)\|^p_p\le C \sum_{a = 0}^\infty W_a^p \cdot (1 +a)^{-1} \sum_{j=0}^\infty \min\{1,2^{-pj+a}\} \le C\,.
\]
This completes the proof.
\end{proof}

\bibliographystyle{abbrv}
\bibliography{bib}

\end{document}